\documentclass[11pt,a4paper]{article}

\usepackage{graphicx}
\usepackage{amsmath,amssymb,amsfonts}
\usepackage{mathrsfs}
\usepackage{xcolor}
\usepackage{subcaption}
\usepackage{hyperref}
\usepackage{geometry}

\newcommand*\bigcdot{\mathpalette\bigcdot@{.3}}

\usepackage{amsthm}
\newtheorem{theorem}{Theorem}
\newtheorem{proposition}[theorem]{Proposition}

\newtheorem{lemma}[theorem]{Lemma}

\theoremstyle{definition}

\theoremstyle{remark}
\newtheorem{remark}{Remark}

\title{Time reversal of reflected Brownian motion with Poissonian resetting}

\author{
	Fausto Colantoni\thanks{Sapienza University of Rome, Rome, Italy; and BCAM--Basque Center for Applied Mathematics, Bilbao, Spain Email: \texttt{fausto.colantoni@uniroma1.it}} \and
	Mirko D'Ovidio\thanks{Sapienza University of Rome, Rome, Italy. Email: \texttt{mirko.dovidio@uniroma1.it}} \and
	Gianni Pagnini\thanks{BCAM--Basque Center for Applied Mathematics, Bilbao, Spain; and Ikerbasque--Basque Foundation for Science, Bilbao, Spain. Email: \texttt{gpagnini@bcamath.org}}
}

\date{\today}

\begin{document}
	
	\maketitle
	
	\begin{abstract}
		In this paper, we study reflecting Brownian motion with Poissonian resetting. After providing a probabilistic description of the phenomenon using jump diffusions and semigroups, we analyze the time-reversed process starting from the stationary measure. We prove that the time-reversed process is a Brownian motion with a negative drift and non-local boundary conditions at zero. Moreover, we further study the time-reversed process between two consecutive resetting points and show that, within this time window, it behaves as the same reflecting Brownian motion with a negative drift, where both the jump sizes and the time spent at zero coincide with those of the process obtained under the stationary measure. We characterize the dynamics of both processes, their local times, and finally investigate elliptic problems on positive half-spaces, showing that the two processes leave the same traces at the boundary.
	\end{abstract}
	
	\bigskip
	
	
	\medskip
	
	
	\section{Introduction}
One-dimensional Brownian motion with Poissonian resetting at a constant rate \( r \) is a well-studied model in the context of stochastic processes with restarts. The resetting mechanism is relevant in search problems, because it allows for transforming an infinite mean first-passage time into a finite one. For an overview of the topic and related developments, we refer to the review \cite{evansreview}.

We consider a Brownian particle that is periodically reset to a fixed point. The resetting occurs at random times following a Poisson process, meaning that the inter-reset times are exponentially distributed. Specifically, at any instant, a reset occurs with constant rate \( r \), independently of the time passed since the last reset.  

We will use the probabilistic definition provided in \cite{sde-resetting}, which characterizes the process as the solution of a stochastic differential equation with jumps. Our focus is on reflecting Brownian motion at zero, with Neumann boundary conditions, subjected to Poissonian stochastic resetting that returns the process to zero, where it then reflects in the positive half-line.  

In this paper, we investigate the time-reversed process of reflecting Brownian motion with Poissonian resetting and its connection to Non-Local Boundary Value Problems (NLBVPs). This is a novel topic in probability theory, where local derivatives are considered inside the domain, while non-local operators, such as Marchaud or Caputo derivatives, are applied at the boundary, see \cite{boncoldovpag2024,colantoni2023non}.  The analysis of time-reversed processes between resetting events is crucial as it 
	addresses a significant conceptual gap in the stochastic thermodynamics of resetting 
	systems, particularly regarding the formulation of fluctuation theorems for Brownian 
	motion with Poissonian resetting \cite{fluctuation}. In this context, Brownian motion 
	with resetting reaches a non-equilibrium steady state (NESS), and we identify how 
	detailed balance is modified in the presence of resetting.

To the best of our knowledge, there are currently no theoretical results for the time reversal of jump-diffusions. Existing results include the time reversal formula for drift-diffusions \cite{reverse-diffusion1,reverse-diffusion2}, studies on reflecting boundaries in this context \cite{reverse-reflected}, and the case of Markov processes with drift and jumps \cite{conforti22}.  

We first focus on the case where the initial distribution coincides with the stationary measure, as in \cite{reverse-jumps}. This choice has a dual interpretation: from a physical perspective, it represents the equilibrium state the process seeks, ensuring consistency between the original and time-reversed processes; from a probabilistic perspective, if the processes are initially distributed as the stationary distribution, then the semigroup of the time-reversed process coincides with that of a Brownian motion with negative drift satisfying the NLBVP.

In addition, we examine the time-reversed process between two consecutive resetting points and provide that the same Brownian motion with negative drift naturally emerges. Moreover, we show that, in distribution, the position of the jumps and the time between jumps in the process governed by the NLBVP coincide with the position reached before the reset and the local time accumulated at zero for the paths between two consecutive resetting points. Therefore, by applying the Markov property, we can study each segment of the paths separately and recover the same conclusions obtained when starting from the stationary distribution.

To completely characterize the boundary behavior, we analyze functionals at zero. Interestingly, while the origin is an attractive point for reflecting Brownian motion with resetting, it acts as a repelling point for Brownian motion with NLBVPs, as trajectories tend to be pushed away. Despite the clear difference in the paths, we show that the law of local times at zero is the same for the processes. This allows us to study the trace processes on the positive half-space and prove that their traces also coincide.

For the convenience of the reader, we provide below a list of the stochastic processes defined throughout the manuscript:

\begin{itemize}
	\item[-] $B$: one-dimensional Brownian motion.
	\item[-] $B^+$: reflecting Brownian motion on the positive half-line.
	\item[-] $N$: Poisson process with rate $r$.
	\item[-] $X$: Brownian motion with Poissonian resetting, see \eqref{sde:BMresetting}.
	\item[-] $X^+$: reflecting Brownian motion with Poissonian resetting, see \eqref{sde:reflectedresetting}.
	\item[-] $\widetilde{B}$: reflecting Brownian motion with drift $-2\sqrt{r}$ on the positive half-line.
	\item[-] $\widetilde{X}$: Brownian motion with drift $-2\sqrt{r}$ and reflection/jumps at zero, see \eqref{pallino}.
	\item[-] $X^R$: time-reversed version of $X^+$.
\end{itemize}

\section{Brownian motion with stochastic resetting on the half-line}
\label{sec:bmrestting}
Let \( B := \{B_t, t \geq 0 \} \) be the one-dimensional Brownian motion with law \( g(t, z) = e^{-z^2/ 4t}/\sqrt{4 \pi t} \), and \( N := \{N_t, t \geq 0 \} \) be an independent Poisson process with intensity \( r>0 \). The idea is to use the Poisson process to formulate a stochastic differential equation with jumps, a relevant reference on this topic is \cite{bass-sde-jumps}. We proceed as in \cite[Formula (4)]{sde-resetting} and define the jump-diffusion process:
\begin{align}
	\label{sde:BMresetting}
	dX_t = dB_t + (x_r - X_t) dN_t , \quad X_0 = x,
\end{align}
where \( x \) is the starting point and \( x_r \) the resetting point, with \( x, x_r \in \mathbb{R} \). Intuitively, when the Poisson process does not jump, \( X \) is simply a Brownian motion, but when a jump occurs, it restarts from the position \( x_r \). The same idea was also used in \cite{pal2017integral}.

Now, we compute the generator by reasoning as in \cite{generators}. Considering a small \( t > 0 \), since the probability that \( N_t \) makes two jumps is \( o(t^2) \), we focus only on the probability of having zero or one jump:
\begin{align*}
	\mathbf{P}(N_t=0) &= e^{-rt} = 1 - rt + o(t^2), \\
	\mathbf{P}(N_t=1) &= rt\,e^{-rt} = rt + o(t^2).
\end{align*}
For a continuous and bounded function \( f \) such that the infinitesimal generator of \( X \) is well-defined, using the notation \( \mathbf{E}_x \) for the expected value with respect to \( \mathbf{P}_x \), where \( x \) is the starting point, we have:
\begin{align}
	\label{conti:generatore}
	\mathbf{E}_x [f(X_t)] &= (1 - rt + o(t^2)) \mathbf{E}_x[f(X_t) \vert N_t=0] + (rt + o(t^2)) \mathbf{E}_x[f(X_t) \vert N_t=1] \notag\\
	&= (1 - rt + o(t^2)) \mathbf{E}_x[f(B_t)] + (rt + o(t^2)) \mathbf{E}_x[f(x_r)]	\notag\\
	&= \mathbf{E}_x[f(B_t)] + rt \left(f(x_r) - \mathbf{E}_x[f(B_t)] \right) + o(t^2).
\end{align}
By Taylor's formula, we know:
\begin{align}
	\label{conti:generatoreBM}
	\mathbf{E}_x[f(B_t)] &= f(x) + \frac{d}{dx} f(x) \mathbf{E}_x[B_t] + \frac{1}{2} \frac{d^2}{dx^2} f(x) \mathbf{E}_x[(B_t)^2] + o(t^{3/2}) 	\notag\\
	&= f(x) + t \frac{d^2}{dx^2} f(x) + o(t^{3/2}),
\end{align}
where we used \( \mathbf{E}_x[B_t] = 0 \), \( \mathbf{E}_x[(B_t)^2] = 2t \), and \( \mathbf{E}_x[(B_t)^3] \sim t^{3/2} \). Combining \eqref{conti:generatore} and \eqref{conti:generatoreBM}, we obtain the infinitesimal generator \( A \) as:
\begin{align}
	\label{generator}
	Af(x) := \lim_{t \downarrow 0} \frac{\mathbf{E}_x [f(X_t)] - f(x)}{t} = \frac{d^2}{dx^2} f(x) + r (f(x_r) - f(x)).
\end{align}
We observe that this infinitesimal generator can be seen as the sum of a diffusive part and a jump part, given by
\begin{align}
	\label{jump-measure}
	r (f(x_r) - f(x))= \int_\mathbb{R} (f(y) - f(x)) J(dy),
\end{align}
with $J(dy)=r \delta_{x_r}(dy)$. It agrees with the theory of \cite{bass-adding}, where jumps are added to continuous stochastic processes.

We now move on to the study of the probability density function. Recall that if \(T_r\) is the last jump for the Poisson process \(N_t\), then \(t - T_r \sim \text{Exp}(r)\). By conditioning on the jumps of \(N\), as in \cite[Section 6.1.1]{sandev-iomin}, such that there are no more jumps after a certain interval, and considering that after a jump the Brownian motion restarts from the position \(x_r\), we obtain
\begin{align}
	\label{density}
	\mathbf{P}_x (X_t \in dy) = e^{-rt} \mathbf{P}_x(B_t \in dy) + \int_{0}^{t} r e^{-rs} \mathbf{P}_{x_r} (B_s \in dy) \, ds,
\end{align}
which coincides with the one already computed in \cite[Formula (5)]{evans2014}.

We turn our attention to the process of interest. Let $B^+:=\{B_t^+, t \geq 0 \}$ be a one dimensional reflected Brownian motion, independent of $N$.  We know that \(B^+\) is a Brownian motion on the positive half-line that reflects instantaneously to \((0, \infty)\) every time it touches zero.
Then, we introduce the process with resetting  similarly to \eqref{sde:BMresetting}
\begin{align}
	\label{sde:reflectedresetting}
	dX_t^+= dB_t^+  - X_t^+ dN_t , \quad X_0^+ =x,
\end{align}
its generator is denoted by \((A^+, {D}(A^+))\), with $A^+=A$ for \(x_r=0\), and by proceeding as in \eqref{density}, we have that we  write the probability density function as
\begin{align}
	\label{density-reflected}
	\mathbf{P}_x (X_t^+ \in dy) &= e^{-rt} \mathbf{P}_x(B_t^+ \in dy) + \int_{0}^{t} r e^{-rs} \mathbf{P}_{0} (B_s^+ \in dy) \, ds \notag\\
	&=e^{-rt} [g(t,x+y) + g(t, y-x) dy] + \int_{0}^{t} 2r e^{-rs} g(s,y) dy\, ds\,
\end{align}
where we have used, \cite[Section 1]{ItoMckean}, that
\begin{align*}
	\mathbf{P}_x(B_t^+ \in dy)= \left(g(t,x+y) + g(t, y-x)\right) dy.
\end{align*}
From the definition via the SDE \eqref{sde:reflectedresetting} and the previous discussion on stochastic resetting, we have that \(X^+\) is a reflecting Brownian motion with Poissonian resetting at zero. Therefore, we expect that, compared to \(B^+\), the process \(X^+\) will tend to reach zero more frequently, depending on the parameter \(r\), and that this will lead to a difference in the functionals at the origin.
\begin{figure}[h]
	\centering
	\includegraphics[width=13cm]{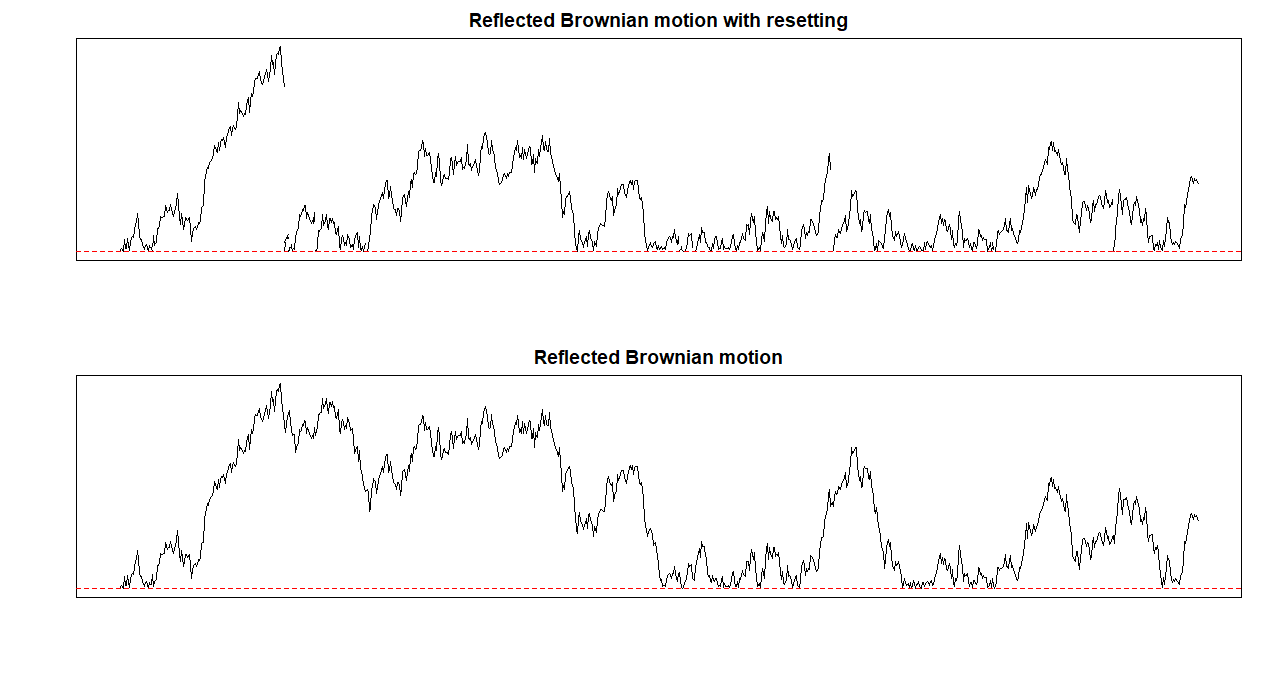} 
	\caption{A possible path for $X^+$ with $r=2$ (above) and $B^+$ (below).}
	\label{fig:resetting}
\end{figure}
We now provide that the boundary conditions remain invariant, even with resetting. We consider the generator $(A^+, {D}(A^+))$
\[
A^+f(x)= \frac{d^2}{dx^2} f(x) + r (f(0) - f(x))\]
with domain
\begin{align}
	\label{DomA}
	{D}(A^+)=\big\{\varphi \in C_b[0,\infty) \cap C^2(0,\infty):\, \varphi^\prime (0)=0 \big\}
\end{align}
\begin{theorem}
	\label{thm:resetting-refl}
	The probabilistic representation of the solution $u \in C((0,\infty) \times [0,\infty))$ of the problem
	\begin{align*}
		\begin{cases}
			\frac{d}{d t} u(t,x)=\frac{d^2}{dx^2} u(t,x) + r (u(t,0) - u(t,x)) \quad &(t,x) \in (0,\infty) \times (0,\infty) \\
			\frac{d}{dx} u(t,x) = 0 \quad    &t>0, x=0\\
			u(0,x) =f(x) \in C_b[0,\infty) \quad &x \in [0,\infty)
		\end{cases}
	\end{align*}
	for $r\geq 0$, is given by
	\begin{align*}
		u(t,x)=\mathbf{E}_x[f (X_t^+)].
	\end{align*}
\end{theorem}
\noindent
The proof is postponed to Section \ref{sec:proofs}.
\section{Non-local boundary value problem}
In this section, we introduce the Non-local Boundary Value Problem (NLBVP) that corresponds to the time-reversal of Brownian motion with Poissonian resetting at zero. By NLBVPs, we refer to equations that are local within the domain but exhibit non-local operators at the boundary.  

This topic was initially introduced by Feller \cite{feller}, who first formulated an integral condition at the boundary. Subsequently, it was analyzed by It\^o and McKean \cite{ItoMckean}, who derived the corresponding probabilistic representation. More recently, the problem has been revisited in \cite{boncoldovpag2024} on the half-line and in \cite{colantoni2023non} on the real line, providing a reinterpretation through non-local operators.  

What emerges from considering these operators at the boundary is a peculiar behavior: the process, upon reaching the boundary, jumps as the last jump of a subordinator.

In all the mentioned cases, the analysis focused on the heat equation without drift. In contrast, here we consider a scenario where drift is present, which is intuitive to understand: looking backward in time at the trajectory of Brownian motion with resetting, we observe Brownian paths that tend to descend toward zero more frequently. This observation suggests that the time-reversal of Brownian motion with resetting is associated with a Brownian motion with a negative drift.

Let \( H^\Phi = \{ H_t^\Phi : t \geq 0 \} \) be a subordinator (see \cite{bertoin1999subordinators} for details). The subordinator \( H^\Phi \) can be characterized by its Laplace exponent \(\Phi\), which satisfies
\begin{align}
	\label{LapH}
	\mathbf{E}_0[\exp(-\lambda H_t^\Phi)] = \exp(-t\Phi(\lambda)), \quad \lambda \geq 0.
\end{align}
To keep the paper self-contained, we provide in the Appendix a short review of subordinators and their main properties which are used in our analysis.
We define the process $L^\Phi=\{L_t^\Phi,\ t \geq 0\}$, with $L_0^\Phi=0$,  as the inverse of $H^\Phi$, that is
\begin{align*}
	L_t^\Phi = \inf \{s > 0\,:\, H_s^\Phi >t \}, \quad t>0.
\end{align*}

Let now introduce the reflecting drifted Brownian motion we will deal with. Let $\widetilde{B}:=\{\widetilde{B}_t, t \geq 0 \}$ be the process related to the infinitesimal generator \((\widetilde{A}^+, {D}(\widetilde{A}^+) ) \), where
\begin{align}
	\label{Atilde+}
	\widetilde{A}^+f(x)=\frac{d^2}{dx^2} f(x) -2\sqrt{r} \frac{d}{dx} f(x),
\end{align}
and
\begin{align*}
	{D}(\widetilde{A}^+):=\big\{\varphi \in C_b[0,\infty) \cap C^2(0,\infty):  \, \varphi^\prime (0)=0 \big\}.
\end{align*}
We have that $\widetilde{B}$ is a reflecting Brownian motion with drift $-2\sqrt{r}$. Furthermore, $\widetilde{\gamma}:=\{\widetilde{\gamma}_t, t \geq 0 \}$ denotes the local time at zero of $\widetilde{B}$.

We define the main process of this section. Let $\widetilde{X}:=\{\widetilde{X}_t, t \geq 0 \}$ be
\begin{align}
	\label{pallino}
	\widetilde{X}_t := \widetilde{B}_t + R^\Psi_{\widetilde{\gamma}_t},
\end{align}
where $\widetilde{B}$ is the reflected Brownian motion with drift mentioned above, $R^\Psi$ is the remaining lifetime
\begin{align*}
	R^\Psi_t:=H^\Psi_{L^\Psi_t} - t := H^\Psi \circ L^\Psi_t - t
\end{align*}
of the subordinator which has the Laplace exponent $\Psi(\lambda)=\lambda + \sqrt{r} -\frac{r}{\lambda + \sqrt{r}}$.
We can easily see that the integral representation of the Bernstein function $\Psi$ is:
\begin{align*}
	\Psi(\lambda)=\lambda + \int_0^\infty (1-e^{-\lambda z}) \Pi^\Psi(dz)= \lambda + r \int_0^\infty (1-e^{-\lambda z}) e^{-\sqrt{r} z} dz.
\end{align*}
Since the subordinator \(H^\Psi\) has a L\'evy measure \(\Pi^\Psi(0, \infty) < \infty\) and a positive drift, it is a strictly increasing process (like \(t\) when there are no jumps) and makes a finite number of jumps in any finite time interval \cite[Theorem 21.3]{sato}.

The paths of the process \(\widetilde{X}\) are those of a Brownian motion with a negative drift as long as \(\widetilde{\gamma}\) remains constant, i.e., while \(\widetilde{X} > 0\). As soon as \(\widetilde{X}\) hits zero, \(\widetilde{\gamma}\) becomes active and starts growing, thereby activating the process \(H^\Psi \circ L^\Psi\): in particular, if \(H^\Psi\) has not yet jumped, \(\widetilde{X}\) reflects in \((0,\infty)\), whereas if \(H^\Psi\) has jumped, then \(\widetilde{X}\) jumps as the last jump of \(H^\Psi\) and subsequently resumes its paths. A complete description, in the absence of drift, is provided in \cite[Section 12]{ItoMckean} and \cite[Section 3.2]{boncoldovpag2024}.
\begin{figure}[h]
	\centering
	\includegraphics[width=10cm]{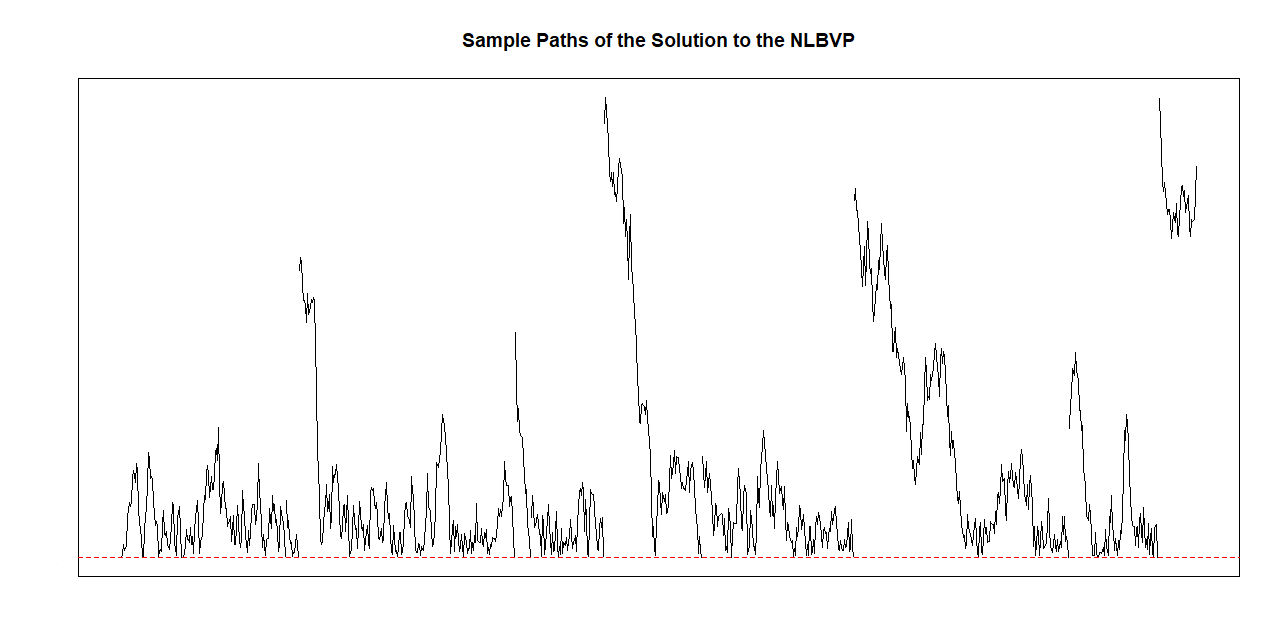} 
	\caption{A possible path for $\widetilde{X}$ with $r=2$.}
\end{figure}
Let us now clarify what we mean by a NLBVP.  
Given a continuous bounded function \( f: (0,\infty) \to (0,\infty) \), we can define the Marchaud-type operator as follows:
\begin{align*}
	\mathbf{D}^\Psi_x f(x):= \int_0^\infty (f(y+x) - f(x)) \Pi^\Psi(dy),
\end{align*}
where \(\Pi^\Psi(dy)= r \exp(-y\sqrt{r}) dy\). We use the term \textit{type} because this operator can be seen as a generalization of the classical Marchaud derivative, see \cite[Section 5.4]{kilbas}, where a more general L\'evy measure is considered, differing from the one associated with the $\alpha-$stable subordinator, with $\alpha \in (0,1)$.

We can therefore state the following theorem, which shows that \(\widetilde{X}\) is governed by the generator \( (\widetilde{A}, {D}(\widetilde{A})) \)
\begin{align*}
	\widetilde{A}f(x) = \frac{d^2}{dx^2} f(x) - 2\sqrt{r} \frac{d}{dx} f(x),
\end{align*}
and
\begin{align*}
	{D}(\widetilde{A}) := \left\{ \varphi \in C_b[0,\infty) \cap C^2(0,\infty): \, \varphi^\prime(0) + \mathbf{D}_x^\Psi \varphi(0) = 0 \right\}.
\end{align*}
\begin{theorem}
	\label{NLBVP}
	The probabilistic representation of the solution $u \in C((0,\infty) \times [0,\infty))$ of the problem
	\begin{align}
		\tag{NLBVP}
		\label{eq:NLBVP}
		\begin{cases}
			\frac{d}{d t} u(t,x)=\frac{d^2}{dx^2} u(t,x) -2\sqrt{r} \frac{d}{dx} u(t,x) \quad &(t,x) \in (0,\infty) \times (0,\infty) \\
			\frac{d}{dx} u(t,x) + \mathbf{D}^\Psi_x u(t,x)=0 \quad    &t>0, x=0\\
			u(0,x) =f(x) \in C_b[0,\infty) \quad &x \in [0,\infty)
		\end{cases}
	\end{align}
	for $r\geq 0$, is given by
	\begin{align*}
		u(t,x)=\mathbf{E}_x[f (\widetilde{X}_t)].
	\end{align*}
\end{theorem}
\noindent
Since the proof is similar to the one provided in \cite[Section 15]{ItoMckean}, we postpone it to the Appendix.
\section{Time reversal under the stationary measure}
\label{sec:stationary}
\subsection{Time reverse process}
We now turn to the central topic: understanding what happens when we reverse the time of a Brownian motion with Poissonian resetting. As we have seen, resetting causes the process to restart from zero, leading to a higher concentration of the paths near the origin. In this sense, the origin can be thought of as an attractive point.

From the perspective of time-reversed trajectories, we expect the process to tend to escape from the origin, making zero a repulsive point in this direction. This repulsion is realized through jumps away from the origin. This idea has already been anticipated in \cite{boncoldovpag2024}.

Let \( X^+ \) denote the Brownian motion with Poissonian resetting, reflected at the origin, as defined in \eqref{sde:reflectedresetting}, and \( X \) denote the free Brownian motion with resetting on \(\mathbb{R}\). From \cite[Formula (2.18)]{evansreview}, it is known that \( X \) admits a stationary distribution given by:
\[
\mu(dx) = \mu(x) dx=\frac{\sqrt{r}}{2} e^{-\sqrt{r} \vert x \vert} \, dx, \quad x \in \mathbb{R},
\]
where \( r > 0 \) is the resetting rate. The stationary state  is simply attained as $t \to \infty$ for \eqref{density}, and we can check that $A^* \mu =0$ in a weak sense, where $A^*$ is the adjoint operator of $A$, defined in \eqref{generator}. We recall that, see \cite[Proposition 9.2]{markov-book}, for the stationary distribution, we have, for any \(t \geq 0\),
\begin{align*}
	\int_{\mathbb{R}} P_t f(x) \mu(dx)= \int_{\mathbb{R}} f(x) \mu(dx),
\end{align*}
where $P_t$ is the semigroup associated to $X$ and we also know that
\begin{align*}
	\mathbf{P}_\mu(X_t \in B)=\mu(B),
\end{align*}
for every Borel set $B \subset \mathbb{R}$, and, in this sense, the stationary distribution is invariant in time.

We expect that, by searching for the limiting distribution also in this case with reflection, we will obtain the stationary distribution.
\begin{lemma}
	The stationary distribution for \(X^+\) is given by
	\[
	\mu^+(dx) = \mu^+(x) dx={\sqrt{r}} e^{-\sqrt{r} x} \, dx, \quad x >0.
	\]
\end{lemma}
\begin{proof}
	From \cite[Proposition 9.2]{markov-book}, we know that an equivalent definition for the stationary distribution is
	\begin{align*}
		\int_0^\infty A^+ f(x) \mu^+(dx)=0.
	\end{align*}
	For \( f \in {D}(A^+)\), we verify that
	\begin{align*}
		\int_0^\infty A^+ f(x) &\mu^+(dx)= \int_0^\infty (f''(x) + r(f(0) - f(x) ) ) \sqrt{r} e^{-\sqrt{r} x}dx\\
		&=\sqrt{r} \big[f'(x) e^{-\sqrt{r} x}\big]^\infty_0 + r \int_0^\infty f'(x) e^{-\sqrt{r} x} + rf(0) - r \sqrt{r} \int_0^\infty f(x) e^{-\sqrt{r} x} dx\\
		&= r\big[f(x) e^{-\sqrt{r} x}\big]^\infty_0 + r \sqrt{r} \int_0^\infty f(x)e^{-\sqrt{r} x} dx + rf(0) - r \sqrt{r} \int_0^\infty f(x) e^{-\sqrt{r} x} dx\\
		&=0,
	\end{align*}
	and this provides that $\mu^+$ is the stationary distribution.
\end{proof}
The stationary distribution is crucial for the reverse process because it represents the long-term behavior of the forward dynamics. This distribution becomes the starting point for the construction of the time-reversed process, ensuring consistency with the equilibrium properties. 

We define the time reversal process 
as \(X^R_t := X^+_{T-t}\) for \(T>0\) constant and \(0 \leq t \leq T\).  In this case, instead of jumping towards zero, the process \(X^R\) jumps away from zero.

We briefly recall the connection between the time-reversed process and the duality with respect to the stationary measure, as established in \cite{nagasawa-quantum,nelson}. First, we define the \(\mu^+\) inner product
\[\langle f , g \rangle_{\mu^+}=\int_0^\infty f(x) g(x) \mu^+(x) dx. \]
\begin{proposition}
	\label{prop:stationary}
	For the time reversed process \(X^R\), we have, for almost every \(x,y>0\),
	\begin{align}
		\label{duality-density}
		\mu^+(dx) \mathbf{P}_x(X^R_t \in dy)= \mu^+(dy) \mathbf{P}_y(X^+_t \in dx).
	\end{align}
	Let \(P^+_t\) and \(P^R_t\) be the semigroup of the process \((X^+_t, t \in [0,T], \mathbf{P}_{\mu^+})\) and of its time reversal \((X^R_t , t \in [0,T], \mathbf{P}_{\mu^+})\), where
	the initial distribution is an invariant measure \(\mu^+\). Then, the following duality formula holds 
	\begin{align}
		\label{duality-semigrups}
		\langle P^+_t f , g \rangle_{\mu^+}= \langle  f , P^R_t g \rangle_{\mu^+},
	\end{align}
	for any bounded measurable \(f\) and \(g\). In addition, by defining \((A^R, {D}(A^R))\) the generator of the process \(X^R\), for any \(f \in {D}(A^+)\) and \(g \in {D}(A^R)\), we have
	\begin{align}
		\label{duality-generator}
		\langle A^+ f , g \rangle_{\mu^+}= \langle  f , A^R g \rangle_{\mu^+}.
	\end{align}
\end{proposition}
\noindent
The proof is postponed to Section \ref{sec:proofs}.

\begin{figure}[ht]
	\centering
	\begin{subfigure}[b]{0.4\textwidth}
		\centering
		\includegraphics[width=\linewidth]{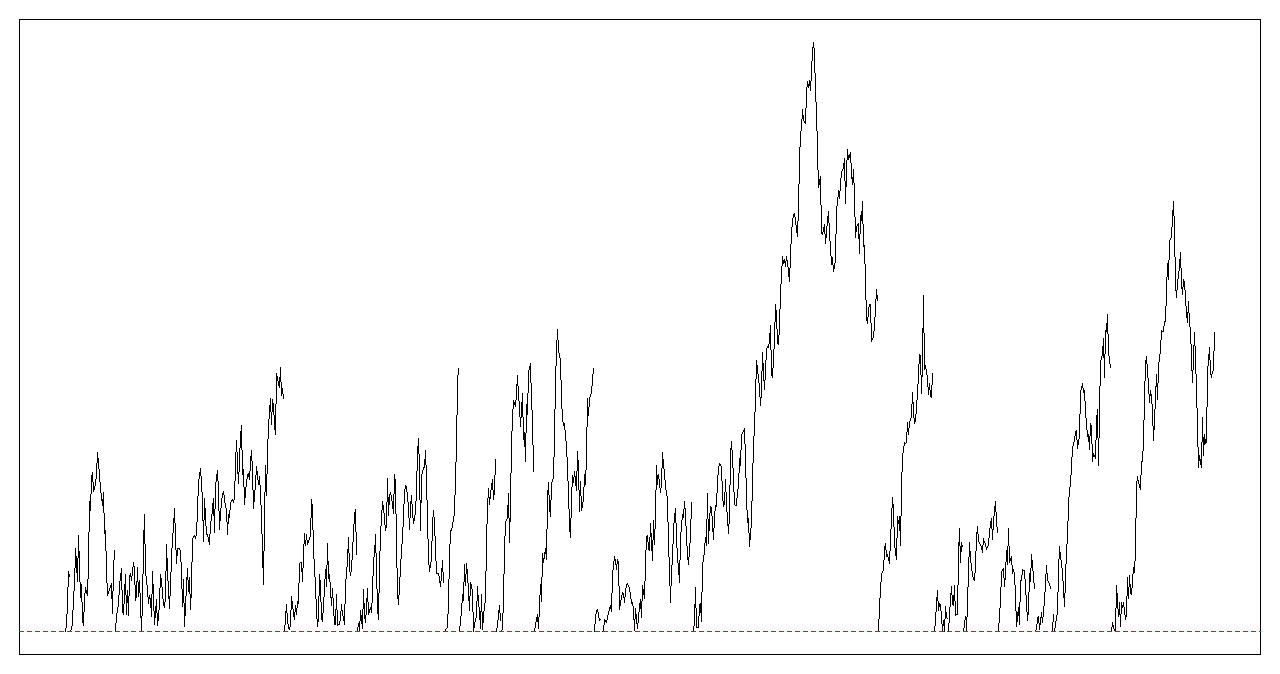}
	\end{subfigure}
	\hspace{1cm}
	\begin{subfigure}[b]{0.4\textwidth}
		\centering
		\includegraphics[width=\linewidth]{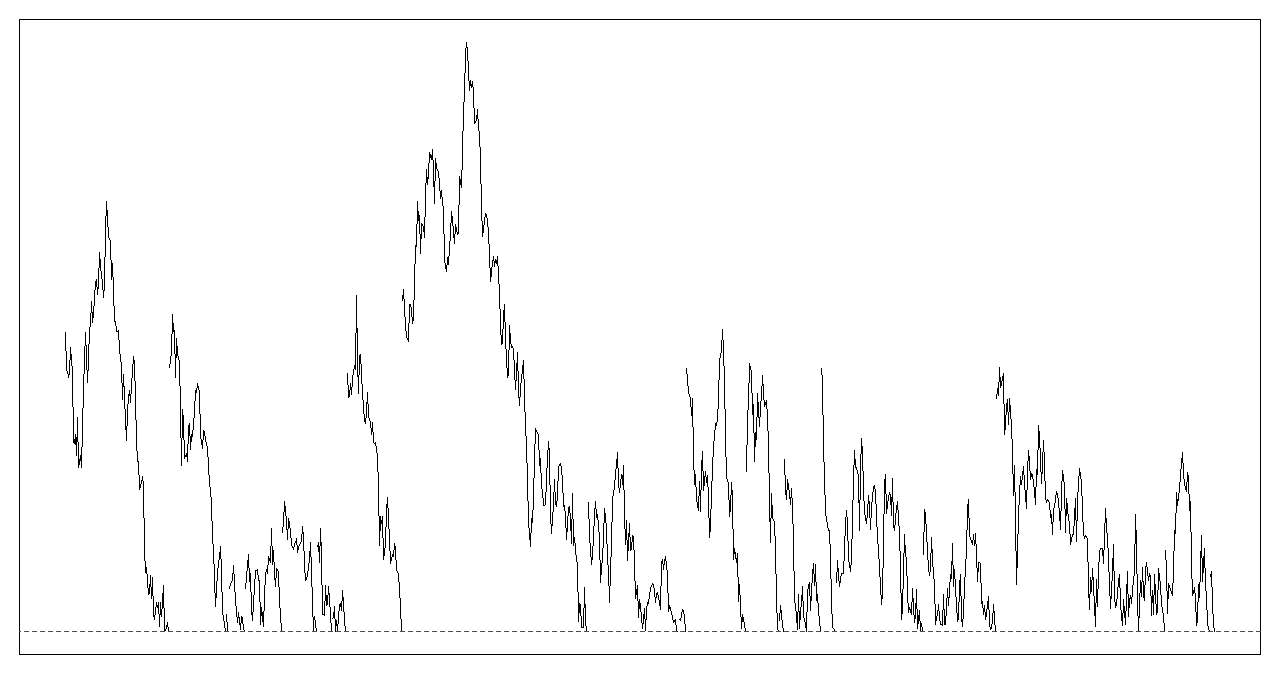}
	\end{subfigure}
	\caption{Comparison between \(X^+\), on the left, and its time reversal, on the right.}
	\label{fig:comparison}
\end{figure}
It is well known that Brownian motion with Poissonian resetting admits a non-equilibrium stationary state, meaning that the time-reversed process does not preserve the same distribution. This  also holds for stable processes with partial resetting, see \cite{ness-stable}. We now present the main result on the time-reversed process, which establishes the 
connection with the NLBVP: when the initial distribution is the stationary measure, 
	the time-reversed process $X^{R}$ is equal in law to $\widetilde{X}$, the 
	Brownian motion with drift $-2\sqrt{r}$ associated with the NLBVP stated in 
	Theorem~\ref{NLBVP}.

\begin{theorem}
	\label{thm:time-reverse}
	The generator \((\widetilde{A}, {D}(\widetilde{A}))\) satisfies the duality formula \eqref{duality-generator}. Then, \((X^R_t , t \in [0,T], \mathbf{P}_{\mu^+})\) is equal in law to \((\widetilde{X}_t , t \in [0,T], \mathbf{P}_{\mu^+})\), where the initial distribution is \(\mu^+\).
	%
\end{theorem}
\noindent
The proof is postponed to Section \ref{sec:proofs}.
\begin{remark}
	Our duality result agrees with that of \cite{nagasawa-reflecting}, with some differences: in our case, we consider a diffusion with jumps, which consequently affects the boundary conditions, rather than a diffusion with drift and reflecting barriers. A more general result, though without boundary conditions, can be found in \cite{nelson}.
\end{remark}
\begin{remark}
	Our result can be viewed as a specific case of \cite[Theorem 2.1]{reverse-jumps} in an unbounded domain,
	where both the generator and the stationary measure are explicitly known. Later, we will detail what the authors refer to as "accessible" and "inaccessible" boundaries in our context. From a physical perspective, we have found that the processes \(X^+\) and \(\widetilde{X}\) share the same equilibrium, so we can reverse one to obtain the other.
\end{remark}
\begin{remark}
	Despite considering different generators, our jump measure coincides with the one analyzed in \cite{conforti22}, where Markov processes with only drift and jumps (without a diffusive component) are studied. The reverse jumping measure we derived satisfies the equation
	\begin{align*}
		\mu^+(dx) \widetilde{J}(x, dy) = \mu^+(dy)J(dx),
	\end{align*}
	where $\widetilde{J}(x, dy)$ is the jumping measure for the process $\widetilde{X}$ and depends on the position $x$. From the definition of $J(dx)=r \delta_0(dx)$, see \eqref{jump-measure}, we have
	\begin{align*}
		\widetilde{J}(x=0, dy)= r e^{-\sqrt{r} y} dy,
	\end{align*}
	this means that from the position \( x = 0 \), the process $\widetilde{X}$ jumps according to the distribution $\widetilde{J}$, which coincides with the one that we have studied as boundary condition in Theorem \ref{thm:time-reverse}.
\end{remark}
\begin{remark}
	For the time reverse, we comprehensively exploit the duality of the generators, which is a direct consequence of having chosen \(\mu^+\) as the initial distribution. Generally, these results can be extended, for transient processes, if the invariant measure is an excessive measure, of the form
	\begin{align*}
		m(B)=\int_0^\infty \int_0^\infty \mathbf{P}_x(X_s \in B) \nu(dx) \, ds,
	\end{align*}
	for every Borel set $B \subset \mathbb{R}$. In these cases, it is possible to study the time-reversed process by taking \(\nu\) as the initial distribution and reversing the paths starting from a cooptional time. For further details, see \cite{nasawa-markov} and \cite[Theorem 4.5, Chapter VII]{revuz-yor}.
\end{remark}

	\begin{remark}
		It is known that standard Brownian motion with Poissonian resetting admits a 
		stationary measure \cite[Formula (2.18)]{evansreview}, namely
		\[
		\mu(dx) = \mu(x)\,dx = \frac{\sqrt{r}}{2} e^{-\sqrt{r}\,|x|}\,dx, 
		\qquad x \in \mathbb{R}.
		\]
		Therefore, the duality techniques, with respect to the stationary measure, developed in Theorem \ref{thm:time-reverse} 
		can also be applied in this setting. The main difference is that the origin is 
		not a boundary point but an interior point, so the time-reversed dynamics is 
		expected to satisfy a transmission condition at zero, with jumps occurring on 
		both sides of the origin.
\end{remark}
\section{Time reversal between two consecutive resetting points}
\label{sec:paths}	
In this section, we study the time-reversed process between two consecutive resetting points. In particular, we initially focus on the generator before the first resetting point. Then, we compute the distribution of the jumps and the local time at zero to show that we can concatenate individual path segments and use the Markov property to extend our result to two consecutive resetting points. Throughout this section, we will no longer assume that the initial distribution is \(\mu^+\), but instead, we will consider \(X^+_0 = 0\). In truth, this assumption is not restrictive; however, it allows the generalization between two consecutive resetting points to emerge naturally.

First, we examine how our results fit into the literature on time-reversed processes. Let \(\{Y_t, 0 \leq t \leq 1\}\) be a diffusion on \(\mathbb{R}\), governed by the stochastic differential equation  
\[
dY_t = b(t, Y_t) dt + \sigma(t, Y_t) dW_t,
\]  
where \(\{W_t, 0 \leq t \leq 1\}\) is a standard Brownian motion on \(\mathbb{R}\). We assume the usual conditions that \( b \) and \( \sigma \) are globally Lipschitz and satisfy the linear growth condition. We observe that, for \(0 \leq t \leq 1\), the Brownian motion \(B_t\), introduced in Section \ref{sec:bmrestting}, is  \(B_t \stackrel{d}{=}\sqrt{2} W_t\). We define the time-reversed process \({Y}_t^R := Y_{1-t}\), for \(0 \leq t \leq 1\). It is clear that \({Y}^R\) is a Markov process, but we cannot conclude a priori that it is a diffusion.  

In \cite{reverse-diffusion1}, the authors provide sufficient conditions for \({Y}^R\) to be a diffusion. Later, under the assumption that \(Y\) has a density for every \(t\), it is shown in \cite{reverse-diffusion2} that this condition is both necessary and sufficient. In particular, from these works, we know that the infinitesimal generator \((\overline{L}_t, {D}(\overline{L_t}))\) of \({Y}_t^R\) is  
\begin{align}
	\label{Lsegnato}
	\overline{L}_t f (x)= &\frac{1}{2} \sigma^2(1-t,x) \frac{d^2}{dx^2} f(x) \,+\notag \\
	&+\left(\frac{1}{p(1-t,x)} \frac{d}{dx} [\sigma^2(1-t,x) p(1-t,x)] - b(1-t,x)\right) \frac{d}{dx} f(x),
\end{align}
and \({D}(\overline{L}_t)=\big\{\varphi \in C^2(\mathbb{R}) \big\}\) for all \(t \in [0,1]\), where \(p\) is the density of \(Y\). In \cite{reverse-reflected}, the author studies the time-reversed process for reflected Brownian motion \mbox{$\{W^+_t, 0 \leq t \leq 1\}$}
\begin{align*}
	dW^+_t= dW_t + \frac{1}{2} dl_t(W^+)
\end{align*}
where \(l(W^+)\) is the local time at zero of \(W^+\), and, by using \cite[Lemma 2.1]{reverse-reflected}, we write the generator of \(\{W^+_{1-t}, 0 \leq t \leq 1\}\) as \((\overline{L}^+_t, {D}(\overline{L}^+_t))\)
\begin{align}
	\label{Lsegnato+}
	\overline{L}_t^+ f(x) = \frac{1}{2} \frac{d^2}{dx^2} f(x) + \left(\frac{1}{p^+(1-t,x)} \frac{d}{dx} p^+(1-t,x)\right) \frac{d}{dx} f(x)
\end{align}
and
\[
{D}(\overline{L}_t^+)=\big\{\varphi \in C_b[0,\infty) \cap C^2(0,\infty):  \, \varphi^\prime (0)=0 \big\},
\]
where \(p^+\) is the density of \(W^+\). As before, our reflecting Brownian \(B^+_t \stackrel{d}{=}\sqrt{2} W^+_t\), for \(0 \leq t \leq 1\).

In all these results, the time reversal is considered over a fixed time interval, namely \([0,1]\). Our goal, however, is to study the problem over a random time interval, which prevents us from directly applying the above theory. We know that the jumps of \(X^+\) are determined by those of \(N\), which is a subordinator with a finite L\'evy measure. Thus, over the interval \([0,t]\), we can enumerate the jumping times \(T_1, T_2, \dots, T_n\), such that for each \(i\), \(T_{i+1} - T_i \sim \text{Exp}(r)\).  Between two consecutive reset times \(T_i\) and \(T_{i+1}\), \(X^+\) evolves as a reflecting Brownian motion started at zero, denoted by \(B^+\). Let us begin by understanding what occurs prior to the first reset.
\begin{theorem}
	\label{thm:paths}
	Let \(T_1\) be the first jumping time of the Poisson process \(N\) and \(B^+\) be the reflected Brownian motion. Then, for \(B^+_0=0\), we have that the time reverse process \(B^+_{T_1 - t}\), for \(0 \leq t \leq T_1\), is governed by the generator \((\widetilde{A}^+, {D}(\widetilde{A}^+))\), introduced in \eqref{Atilde+}.
	So, \(B^+_{T_1-t}\) has the same generator of \(\widetilde{B}_t\) for \(0 \leq t \leq T_1\).
\end{theorem}
\noindent
The proof is postponed to Section \ref{sec:proofs}.

So far, we have obtained that, since before \( T_1 \) the process \( X^+ \) behaves as \( B^+ \), reversing \( X^+_{T_1 - t} \) yields the process \( \widetilde{B}_t \), for \(0\leq t < T_1\). Furthermore, by \cite[Formula 1.0.4, page 333]{BorodinSalminen}, the position of a reflecting Brownian motion \(B^+\) started at zero and evaluated at an exponential time \(T_1\) is distributed as  
\[
p(dy) = \sqrt{r} e^{-\sqrt{r} y} \, dy.
\]
This describes the distribution of the position reached by \(X^+\) immediately before the first resetting. At time \( T_1 \), we know that the process \( X^+ \) is reset to zero, i.e., \( X^+_{T_1} = 0 \). However, since it is a Markov process, we can restart the paths from this point without considering the past.  

Thus, we interpret the interval \([T_1, T_2)\) as analogous to \([0, T_1)\), and for \(t \in [T_1, T_2)\), \({X}^+_t\) behaves as a reflecting Brownian motion \(B^+_t\) starting from zero and evolving until an exponential time \( \text{Exp}(r) \), since \( T_2 - T_1 \) represents the interarrival time between two consecutive jumps of the Poisson process \( N \), meaning that \( T_2 - T_1 \sim \text{Exp}(r) \). 

Therefore, in this interval as well, we obtain that the time-reversed process is \( \widetilde{B} \) and that the position reached at \( T_2 \) is distributed according to \( p(dy) \).  

Additionally, we can derive another property of \( X^+ \). Let \(\gamma^+\) denote the local time of a reflecting Brownian motion \(B^+\) with \(B^+_0=0\). Then, the local time accumulated at zero by \(X^+\) between two resetting points \(T_1\) and \(T_2\), denoted by \(\gamma^+_{(T_{2} - T_1)}\), follows an \(\text{Exp}(\sqrt{r})\) distribution. In fact, we obtain  
\begin{align*}
	\mathbf{P}_0 (\gamma^+_{T_2-T_1} \in dw) &= \int_0^\infty \mathbf{P}_0 (\gamma^+_{s} \in dw) \mathbf{P}(T_2-T_1 \in ds)  \\
	&= r \int_0^\infty e^{-rs} \mathbf{P}_0 (\gamma^+_{s} \in dw) \\
	&= \sqrt{r} e^{-w \sqrt{r}} dw.
\end{align*}  

We can then recursively apply the Markov property at each resetting point \( T_i \). Since, by definition, \( X^+_t \) behaves as \( B^+_t \) between two consecutive resetting points, that is for  \(t \in [T_i, T_{i+1})\), with \( i = 0,1, \dots \) and \( T_0 = 0 \), where \( B^+_t \) starts from zero and \( T_{i+1} - T_i \sim \text{Exp}(r) \) as the consecutive interarrival times of the Poisson process \(N\), we consistently obtain the following three properties:  

1. The time-reversed process \( X^+_{T_{i+1} - t} \stackrel{d}{=} \widetilde{B}_t \) for \( t \in [T_i, T_{i+1}) \).  

2. The position reached by \( X^+ \) at \( T_{i+1} \), starting from zero at time \( T_i \), follows the distribution \( p(dy) \sim \text{Exp}(\sqrt{r}) \).  

3. The local time at zero, for \( t \in [T_i, T_{i+1}) \), is distributed as \( \text{Exp}(\sqrt{r}) \).

\noindent
Let us now see how these three points are consistent with the process \( \widetilde{X} \). First, \(\widetilde{X}\), outside from the jumping point, is \(\widetilde{B}\), as we see from the definition \eqref{pallino}. But, in zero, \(\widetilde{X}\) is governed by the subordinator \(H^\Psi\) with L\'evy measure \(\Pi^\Psi(dz)= r e^{-z\sqrt{r}} dz\): it either jumps according to the last jump of the subordinator or reflects if no jump has occurred yet.

The i.i.d. jump times \(e_1, e_2, \dots\) of the subordinator are characterized by \(\mathbf{P}(e_1 > t)=\exp(-t \, \Pi^\Psi(0,\infty))= \exp(-t \sqrt{r})\), so they follow an \(\text{Exp}(\sqrt{r})\) distribution. Since the process \(\widetilde{X}\) at zero evolves as the subordinator \(H^\Psi\), and the local time accumulated at zero by \(X^+\) between resetting points is also \(\text{Exp}(\sqrt{r})\) distributed, we conclude that this local time is the time needed for a jump to occur in the time-reversed process.
Moreover, as shown in \cite[Section 11]{ItoMckean}, the distribution of the i.i.d. jumps \(l_1, l_2, \dots\) of the subordinator is given by  
\[
\mathbf{P}(l_i \in dy) = \frac{1}{\Pi^\Psi(0,\infty)} \Pi^\Psi(dy)= \sqrt{r} e^{-\sqrt{r} y} \, dy,
\]
which coincides with the distribution of the position of a reflecting Brownian motion evaluated at exponential times.

Therefore, by considering the paths piece by piece between resetting points, the jumps of \(\widetilde{X}\) follow the same law as the position of \(X^+\) immediately before a reset. This ensures that \(\widetilde{X}\) "jumps to the correct position," consistently with the time-reversed dynamics of a reflecting Brownian motion with Poissonian resetting.

In summary, the process \( \widetilde{X} \), outside the jumping points, behaves like \( \widetilde{B} \), which is the time-reversed process of \( X^+ \) in the intervals \([T_i, T_{i+1})\). Furthermore, the time that \( \widetilde{X} \) spends at zero before jumping is distributed in the same way as the local time of \( X^+_t \) for \( t \in [T_i, T_{i+1}) \), and the position of the jumps is distributed in the same way as the jump size of the resetting process that occurs at \( X^+_{T_i} \). In this sense, we note the connection between the time-reversed paths of \( X^+ \) and the paths of \( \widetilde{X} \) in the intervals \([T_i, T_{i+1})\), which we can concatenate piece by piece thanks to the Markov property.

We conclude by observing that the boundary behavior is also reversed. While \( X^+ \) exhibits a tendency to accumulate near the boundary due to resetting, its time-reversed counterpart \( \widetilde{X} \) instead attempts to escape upon reaching it. The negative drift ensures that \( \widetilde{X} \) also tends towards zero before jumping away.
\begin{figure}[h]
	\centering
	\includegraphics[width=4.5cm]{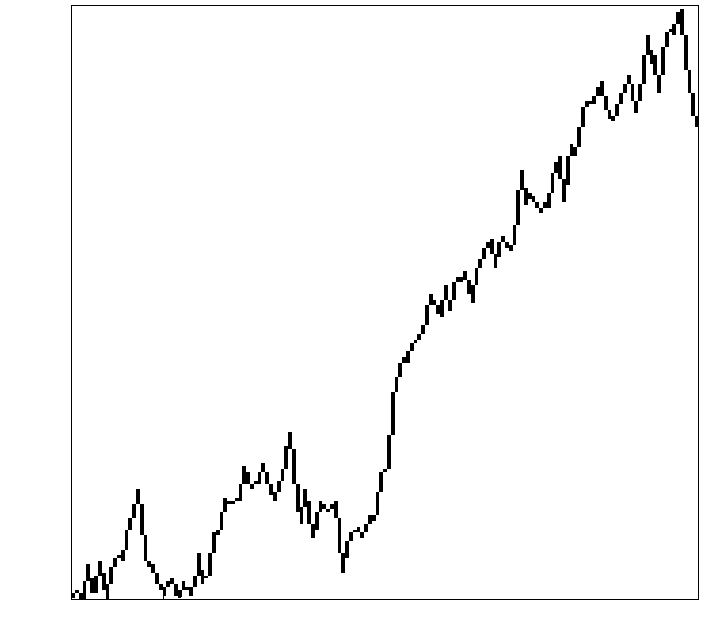} 
	\caption{A possible path for $X^+$ between two resetting times.}
	\label{fig:escursione}
\end{figure}
\section{Inverse of local time at the origin}
\label{sec:local}
In this section, we focus on the local time at the origin and we show that the one of the Brownian motion with resetting \(X^+\) and the one of the reversed process $\widetilde{X}$ are equal in law. This ensures that, despite the opposite boundary behavior, where \(X^+\) tends to concentrate more near zero, while \(\widetilde{X}\) tends to jump away from it, the time they spend at the boundary is the same.
For the stochastic resetting, this subject has also been explored in \cite[Section 3]{singh2022}, with a focus up to the first-passage time. The general topic of local time for diffusions with stochastic resetting has already been addressed in \cite{localtime}, where explicit moments and densities can also be found. We will focus more on the relation with subordinators due to the Markovian nature of the process.

The connection between subordinators, their inverses, and local times is well known: given a Markov process, its local time (at a regular point) is the inverse of a subordinator \cite[Theorem 2.3]{BlumenthalGetoor}. Let \( L^\Phi \) be the local time of a Markov process at zero. We know that the symbol characterizing the subordinator (which is the inverse of the local time) is given by \cite[Proposition 4, Section V.I]{bertoinlevy}
\begin{align}
	\label{Phi:Markov-LocalTime}
	\Phi(\lambda) = \frac{1}{u^\lambda(0)},
\end{align}
where $u^\lambda(\cdot)$ is the density of the resolvent associated to the Markov process. Let us now see this connection in the case of a process with resetting.
\begin{theorem}
	\label{tm:inverselocaltime}
	The inverse of the local time at zero of $X^+$, started at zero, is a subordinator with Laplace exponent
	\begin{align}
		\label{Laplace-exponent}
		\Phi(\lambda)=\frac{\lambda}{\sqrt{\lambda + r}}, \quad \lambda >0.
	\end{align}
\end{theorem}
\begin{proof}
	That the inverse of the local time is a subordinator is guaranteed by the results stated previously, since resetting preserves the Markovian property of the process \(X^+\). We check its Laplace exponent.
	By applying the Laplace transform  at \eqref{density-reflected}, we have, for $\lambda>0$,
	\begin{align*}
		u^\lambda(0)&= \int_0^\infty e^{-\lambda t } \left(e^{-rt} [g(t,0) + g(t, 0)] + \int_{0}^{t} 2r e^{-rs} g(s,0) \, ds\right) dt\\
		&=\frac{1}{2}\frac{1}{\sqrt{\lambda + r}}+ \frac{1}{2}\frac{1}{\sqrt{\lambda + r}} + \frac{r}{\lambda}\frac{1}{\sqrt{\lambda + r}} \\
		&=\frac{1}{\sqrt{\lambda + r}}  \left(1+\frac{r}{\lambda}\right)\\
		&=\frac{\sqrt{\lambda+r}}{\lambda}.
	\end{align*}
	From \eqref{Phi:Markov-LocalTime}, we have our claim.
\end{proof}
We already know that the function \(\Phi\) found is the Laplace exponent of a subordinator (\cite[Theorem 5.2]{schilling2012bernstein}). We also add that since \(\sqrt{\lambda + r}\) is a complete Bernstein function (i.e., a Bernstein function for which the L\'evy measure \(\Pi^\Phi\) has a completely monotone density), it follows, from \cite[Proposition 7.1]{schilling2012bernstein} that \(\Phi(\lambda) = \frac{\lambda}{\sqrt{\lambda + r}}\) is also a complete Bernstein function, specifically its L\'evy measure is (\cite[Section 16.2, Case 5]{schilling2012bernstein})
\begin{align*}
	\Pi^\Phi(dz)= \frac{e^{-rz} (2rz+1)}{2 \sqrt{\pi} z^{3/2}}\, dz.
\end{align*}
Moreover, for \( r = 0 \), we obtain simply that $\Phi(\lambda)=\sqrt{\lambda}$ and $\Pi^\Phi(dz)=\frac{1}{2 \sqrt{\pi} z^{3/2}}$, which coincide with the well-known properties of reflected Brownian motion.

We now move on the local time at zero of the process $\widetilde{X}$, but first we need this preliminary result.
\begin{theorem}
	\label{tm:local-time-pallino}
	The local time at zero of the process $\widetilde{X}$, started at zero, is given by $L^\Psi \circ \widetilde{\gamma}$.
\end{theorem}
\noindent
Since this result is a generalization of \cite[Section 14]{ItoMckean}, we postpone the proof to the Appendix.

Let us now investigate how these two local times are related.
\begin{theorem}
	\label{tm:same-local-times}
	Also the inverse of the local time at zero of $\widetilde{X}$, started at zero, is a subordinator with Laplace exponent
	\begin{align*}
		\Phi(\lambda)=\frac{\lambda}{\sqrt{\lambda + r}}, \quad \lambda >0.
	\end{align*}
	When the processes \(X^+\) and \(\widetilde{X}\) start at zero, they share, in distribution, the same local time at zero.
\end{theorem}
\begin{proof}
	By integrating \eqref{drift:joint} in $dy$, we obtain, for $\lambda>0$
	\begin{align*}
		\int_0^\infty e^{-\lambda t} \mathbf{P}_0(\widetilde{\gamma}_t \in du)dt= \frac{1}{\sqrt{\lambda + r} + \sqrt{r}} e^{-w (\sqrt{\lambda + r} - \sqrt{r}) }=\frac{\sqrt{\lambda + r} - \sqrt{r}}{\lambda} e^{-w (\sqrt{\lambda + r} - \sqrt{r}) },
	\end{align*}
	where in the last equality we recognize, by using \eqref{symbol:L}, that $\widetilde{\gamma}$ has the law of the inverse of a subordinator with Laplace exponent $\sqrt{\lambda + r} - \sqrt{r}$. From Theorem \ref{tm:local-time-pallino} we know that the local time of $\widetilde{X}$ is given by $L^\Psi \circ \widetilde{\gamma}$, which are independent, then, for $\lambda>0$,
	\begin{align*}
		\int_0^\infty e^{-\lambda t} \mathbf{P}_0 (L^\Psi \circ \widetilde{\gamma} _ t \in dx) dt &= \int_0^\infty e^{-\lambda t} \int_0^\infty \mathbf{P}_0 (L^\Psi_s \in dx) \mathbf{P}_0(\widetilde{\gamma}_t \in ds) \, dt\\
		&= \int_0^\infty \frac{\sqrt{\lambda + r} - \sqrt{r}}{\lambda} e^{-s (\sqrt{\lambda + r} - \sqrt{r})} \mathbf{P}_0 (L^\Psi_s \in dx) ds\\
		&= \frac{\sqrt{\lambda + r} - \sqrt{r}}{\lambda}  \frac{\Psi(\sqrt{\lambda + r} - \sqrt{r})}{\sqrt{\lambda + r} - \sqrt{r}} e^{-x \Psi(\sqrt{\lambda + r} - \sqrt{r})},
	\end{align*}
	where we have used \eqref{symbol:L}.  We simply observe that
	\begin{align*}
		\Psi(\sqrt{\lambda + r} - \sqrt{r})= \sqrt{\lambda + r} - \sqrt{r} + \sqrt{r} - \frac{r}{\sqrt{\lambda + r} - \sqrt{r} + \sqrt{r}} = \sqrt{\lambda + r} - \frac{r}{\sqrt{\lambda + r}}= \frac{\lambda}{\sqrt{\lambda + r}}.
	\end{align*}
	Then, we conclude from
	\begin{align*}
		\int_0^\infty e^{-\lambda t} \mathbf{P}_0 (L^\Psi \circ \widetilde{\gamma} _ t \in dx) dt&=  \frac{\lambda}{\sqrt{\lambda + r}} \frac{1}{\lambda} e^{-x \frac{\lambda}{\sqrt{\lambda + r}}}\\
		&=  \frac{\Phi(\lambda)}{\lambda} e^{-x\Phi(\lambda)},
	\end{align*}
	that $L^\Psi \circ \widetilde{\gamma}$, which is the local time at zero of $\widetilde{X}$, is the inverse of a subordinator with Laplace exponent $\Phi(\lambda)= \frac{\lambda}{\sqrt{\lambda + r}}$, which holds also for $X^+$.
\end{proof}
\section{Trace process}
\label{sec:trace}
In this section, we focus on trace processes and, in particular, examine how the inverse of the local times for Brownian motion with resetting and for the reverse process prove to be useful. We begin with the case of simple Brownian motion.

Let us introduce the positive half-space $H=\mathbb{R}\times [0,\infty)$ with boundary $\partial H = \mathbb{R} \times \{0\}$. As mentioned previously, let $B$ be a Brownian motion on $\mathbb{R}$ and $B^+$ an independent reflected Brownian motion on $[0,\infty)$. Then, we have that $(B, B^+)$ is the Brownian motion on $H$. We denote by $\gamma^+$ the local time of $B^+$ at zero and $\gamma^{-1}$ its right inverse. From our construction, we know that $\gamma^+$ coincides with $L^\Phi$ when $r=0$ and $\gamma^{-1}$ is a $1/2-$stable subordinator with symbol $\Phi(\lambda)=\sqrt{\lambda}$. The trace left by $(B, B^+)$ on $\mathbb{R} \times \{0\}$ is $T:=B \circ \gamma^{-1}$ and it is a Cauchy process, since it is a Brownian motion subordinated with an independent $1/2-$stable subordinator. Let us explore the connection with partial differential equation theory.

First, we consider
\begin{align*}
	\begin{cases}
		\Delta u(x,y)=0 \quad & \text{in } H\\
		u(x,0) =f(x)  \quad &\text{on } \partial H
	\end{cases}
\end{align*}
for $f$ smooth, for example for $f \in \mathcal{S}(\mathbb{R})$ in the Schwartz space, and $\Delta$ the Laplacian in two dimensions. We define the Dirichlet-to-Neumann operator $K$
\begin{align}
	\label{D-to-N}
	K f(x)= \partial_y u(x,y) \big\vert_{y=0},
\end{align}
which is the generator of the trace process $T$. By applying the Fourier transform $\mathcal{F}_x$ in the first variable at
\begin{align*}
	\Delta u(x,y)=0 \quad & \text{in } H,
\end{align*}
we obtain, for $\xi \in \mathbb{R}$ and \(\mathcal{F}_x u (\xi, y)= \int_\mathbb{R} e^{-i x \xi} u(x,y) dx\),
\begin{align*}
	- \xi^2 \mathcal{F}_x u(\xi, y) + + \frac{\partial^2}{\partial y^2} \mathcal{F}_x u(\xi,y)=0.
\end{align*}
Since we are looking for a bounded solution, we have 
\begin{align*}
	\mathcal{F}_x u (\xi, y)= e^{-\vert \xi \vert y} \mathcal{F}_x u(\xi, 0)= e^{-\vert \xi \vert y} \widehat{f}(\xi),
\end{align*}
where $\widehat{f}$ is the Fourier transform of $f$. Then, the Fourier transform of the Dirichlet-to-Neumann \eqref{D-to-N} is
\begin{align*}
	\mathcal{F}_x K f (\xi)= - \vert \xi \vert \widehat{f}(\xi),
\end{align*}
and from this, we conclude $K= -\sqrt{-\Delta}$, in the Phillips' sense, such that $K^2=- \Delta$ which, for the one dimensional case, is $K = - \sqrt{-\partial^2_{xx}}$.
This aligns with our expectations: Indeed, we know that, for the one-dimensional Cauchy process $B \circ \gamma^{-1}_t$, the Fourier symbol is $e^{-t \vert \xi \vert}$.

Let us now analyze how the traces left on $\partial H$ change when resetting is introduced. Given the importance of the local times at zero, what we predict is that the traces left by $X^+$ and $\widetilde{X}$ remain the same, as both share the same local time at zero. 

Consider \((B_1, X^+)\) and \((B_2, \widetilde{X})\) on \( H \cup \partial H \), where \(B_1,B_2\) are two independent Brownian motions on \(\mathbb{R}\), independent also of $B^+$, with \( X^+_0 = \widetilde{X}_0 = 0 \). We define \(\ell^+\) as the local time of \( X^+ \) at zero, and \(\widetilde{\ell}\) as the local time of \( \widetilde{X} \) at zero. 

Then, similarly to what we have seen for reflecting Brownian motion, we define the trace processes on \(\partial H\) as follows:
\[
T_1 := B_1 \circ (\ell^+)^{-1} \quad \text{and} \quad T_2 := B_2 \circ (\widetilde{\ell})^{-1},
\]
where the notation \(-1\) denotes the inverse process. 

The following theorem holds.

\begin{theorem}
	\label{thm:trace}
	The trace processes \(T_1\) and \(T_2\) on \(\partial H\) have the same law.
	In particular, given the problems for $f \in \mathcal{S}(\mathbb{R})$
	\begin{align}
		\tag{P1}
		\label{P1}
		\begin{cases}
			\Delta u_1(x,y) + r (u_1(x,0) - u_1(x,y))=0 \quad & \text{in } H\\
			u_1(x,0) =f(x)  \quad &\text{on } \partial H
		\end{cases}
	\end{align}
	with the Dirichlet-to-Neumann operator $K_1 f(x):= \partial_y u_1(x,y) \vert_{y=0}$, and
	\begin{align}
		\tag{P2}
		\label{P2}
		\begin{cases}
			\Delta u_2(x,y) -2 \sqrt{r} \frac{\partial}{\partial y} u_2(x,y)=0 \quad & \text{in } H\\
			u_2(x,0) =f(x)  \quad &\text{on } \partial H
		\end{cases}
	\end{align}
	with the boundary operator $K_2 f(x):= \partial_y u_2(x,y)\vert_{y=0} + \mathbf{D}_y^\Psi u_2(x,y)\vert_{y=0}  $, we have that pointwise
	\begin{align*}
		K_1 f(x) = K_2 f(x).
	\end{align*}
\end{theorem}
\noindent
The proof is postponed to Section \ref{sec:proofs}.
\begin{figure}[h!]
	\centering
	\includegraphics[width=13.5cm, height=9cm]{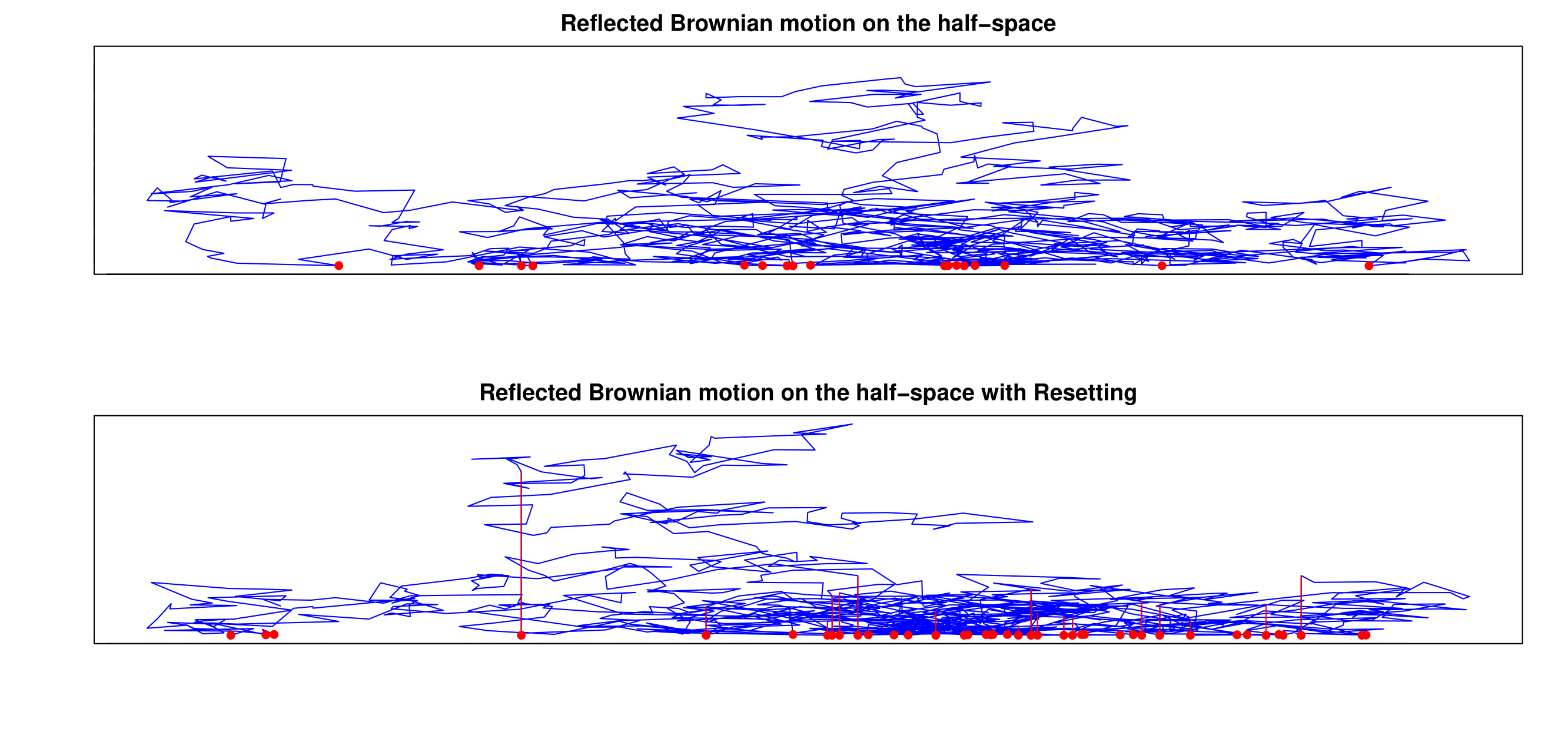} 
	\caption{A possible path for the trace on $\mathbb{R} \times \{0\}$ for the reflected Brownian motion $(B, B^+)$ (above) and for the reflected Brownian motion with stochastic resetting $(B, X^+)$ with $r=2$ (below).}
	\label{fig:tracce}
\end{figure}
\begin{remark}
	It is interesting to note that, as in the case studied in \cite{kwasnicki}, the trace process here is also a L\'evy process with completely monotone jumps, for which the generator has an explicit representation.
\end{remark}

\section{Discussion and Conclusion}
In this work we have analyzed the time reversal of reflecting Brownian motion with Poissonian resetting. Our main results combine both a rigorous mathematical characterization and a physical interpretation.
	\\
	\textbf{Mathematical summary.}\\
	(1) \textit{Time reversal and stationary measure (Section \ref{sec:stationary})}: Starting from the stationary distribution of the process, which corresponds to a nonequilibrium steady state (NESS), we proved that the time-reversed dynamics is a Brownian motion with negative drift, subject to non-local boundary conditions (NLBVPs). This change of law under time reversal reflects the fact that the stationary state is not a thermal equilibrium.\\
	(2) \textit{Time reversal between two consecutive resets (Section \ref{sec:paths})}: We also studied the paths between two resetting events, without assuming stationarity. In this case, each continuous part of the trajectory can be analyzed as a classical reflecting diffusion. By the Markov property, these pieces can be concatenated to reconstruct the global reversed process. In this way we recover the same time-reversed dynamics described in point (1): a Brownian motion with negative drift, subject to NLBVPs. This approach is new, because it does not rely on the existence of a stationary measure and could be extended to non-Poissonian resetting protocols.\\
	(3) \textit{Boundary behavior (Sections \ref{sec:local} and \ref{sec:trace})}: Although the role of the boundary is reversed (attractive in the forward process, repulsive in the reversed one), we showed that the local time at the origin has the same law in both processes. Moreover, the traces left on the positive half-space coincide in distribution, a fact also confirmed by solving the associated elliptic problem with non-local boundary conditions.
	\\
	\textbf{Physical interpretation.}\\
	Time-reversed trajectories are of interest in statistical physics within the framework of fluctuation theorems, which compare forward and backward path probabilities and are used to investigate entropy production at the trajectory level \cite{mori2023entropy} and thermodynamic costs related with irreversibility \cite{olsen2024thermodynamic}. Our findings fit within this framework.
	\\
	It is known that reset processes may result into unidirectional processes, which means they have no time-reversed equivalent. This is the case also of the simple Brownian motion with Poissonian resetting that reaches indeed a non-equilibrium steady state (NESS). Here, we provide for our process how a detailed balance with respect to the stationary measure can be restored.
	\\
	In fact, given a time-series corresponding to an experimental trajectory that undergoes forward-in-time according to a reflecting Brownian motion with Poissonian resetting, we show that such a trajectory is statistically equivalent to a time-series corresponding to a trajectory of a Brownian motion with a negative drift combined with random jumps at the boundary. Thus, this last process identifies the precise backward stochastic dynamics (including the non-local boundary terms) for the reflected and reset Brownian motion.
	\\
	Hence, by making explicit the forward and backward path-laws we have a recipe for how an experiment could be run in either temporal direction: knowing the backward dynamics allows one to reconstruct which dynamics would have to be implemented to reproduce backward-time paths.
	\\
	In our case, in the time-reversed distribution a negative drift term appears, a common characteristic of reflected diffusions \cite{reverse-reflected}, and the non-local boundary conditions correspond, by time inversion, to reflections and resetting.\\
	\textbf{Outlook.}\\
	This work combines a rigorous mathematical framework for reflecting Brownian motion with Poissonian resetting with the physical interpretations discussed above. The pathwise approach developed in Section \ref{sec:paths} is in fact more general: it can be adapted to any resetting protocol, because it only requires the distribution of reset times. The Poissonian case is special since the inter-reset times are exponential, which ensures full Markovianity and allows a complete mathematical treatment. For non-Poissonian resetting the same concatenation method remains applicable, but explicit computations depend on the knowledge of the inter-arrival time law.
	\\
	The reflecting condition at the origin plays a key role in our analysis. Mathematically, it guarantees that trajectories remain confined in the half-line; physically, it can be interpreted as an impenetrable barrier, which forces the system to accumulate “local time” at the boundary instead of crossing it. This mechanism clarifies the role of the origin both in the forward process and in its time-reversed version. The same methodology could also be applied to different boundary conditions: for instance, partially absorbing or sticky boundaries correspond to different ways of weighting or delaying the local time at zero. In the sticky case one would have to take into account not only the diffusive excursions but also the time intervals where the process remains stuck at the boundary.
	\\
	Finally, the choice of resetting to the origin is natural for our setting, since in the time-reversed dynamics it leads to a well-defined structure of reflections and jumps described by the NLBVPs. If the reset location were not fixed but chosen randomly, according to a distribution, the situation would change: the reversed process would involve jumps from random positions, and the resulting boundary conditions would no longer be purely of the non-local type considered here. Such generalizations are left for future investigation.

	\section{Proofs of main results}
	\label{sec:proofs}
	\begin{proof}[Proof of Theorem \ref{thm:resetting-refl}]
		First of all, we recall that, for $\lambda >0$
		\begin{align*}
			\int_0^\infty e^{-\lambda t} g(t,x) dt &= \frac{1}{2} \frac{e^{-\vert x \vert \sqrt{\lambda}}}{\sqrt{\lambda}}.
		\end{align*}
		We use it, together with \eqref{density-reflected}, to calculate the resolvent of the solution, for $\lambda >0$, 
		\begin{align*}
			&\int_{0}^\infty e^{-\lambda t} u(t,x) dt = \int_{0}^\infty e^{-\lambda t} \mathbf{E}_x[f(X^+_t)] dt\\
			&= \int_{0}^\infty \int_{0}^\infty  e^{-\lambda t} f(y) \mathbf{P}_x (X_t^+ \in dy) dt\\
			&= \int_0^\infty f(y) \int_0^\infty e^{-\lambda t} \left[e^{-rt} (g(t,x+y) + g(t, y-x) ) + \int_{0}^{t} 2r e^{-rs} g(s,y) ds\,\right] dt\, dy\\
			&=\int_0^\infty f(y) \left[\frac{1}{2} \frac{e^{-(x+y) \sqrt{\lambda + r}}}{\sqrt{\lambda + r}} + \frac{1}{2} \frac{e^{-\vert y-x \vert \sqrt{\lambda + r}}}{\sqrt{\lambda + r}}\right] dy + \frac{r}{\lambda} \int_{0}^\infty f(y) \frac{e^{-y \sqrt{\lambda + r}}}{\sqrt{\lambda + r}} dy.
		\end{align*}
		For simplicity, we use the notation 
		\begin{align*}
			I_1 &:= \int_0^\infty f(y) \frac{1}{2} \frac{e^{-(x+y) \sqrt{\lambda + r}}}{\sqrt{\lambda + r}} dy\\
			I_2 &:= \int_{x}^{\infty}  f(y) \frac{1}{2} \frac{e^{-(y-x) \sqrt{\lambda + r}}}{\sqrt{\lambda + r}} dy\\
			I_3 &:=\int_{0}^{x} f(y) \frac{1}{2} \frac{e^{-(x-y) \sqrt{\lambda + r}}}{\sqrt{\lambda + r}} dy.
		\end{align*}
		Thus, by decomposing the absolute value, we obtain
		\begin{align*}
			\widetilde{u}(\lambda, x)&:= \int_{0}^\infty e^{-\lambda t} u(t,x) dt \\
			&= I_1+I_2+I_3 +  \frac{r}{\lambda} \int_{0}^\infty f(y) \frac{e^{-y \sqrt{\lambda + r}}}{\sqrt{\lambda + r}} dy,
		\end{align*}
		for which
		\begin{align*}
			\frac{d}{dx} \widetilde{u}(\lambda, x) \vert_{x=0} &= - \sqrt{\lambda + r} \int_{0}^\infty \frac{f(y)}{2} \frac{e^{-y \sqrt{\lambda + r}}}{\sqrt{\lambda + r}} dy - \frac{1}{2\sqrt{\lambda + r}} f(0)\, + \\ & +\,\sqrt{\lambda + r} \int_{0}^{\infty} \frac{f(y)}{2} \frac{e^{-y \sqrt{\lambda + r}}}{\sqrt{\lambda + r}} dy 
			\, + \frac{1}{2\sqrt{\lambda + r}} f(0)\\
			&= 0,
		\end{align*}
		Now, we move on the diffusion equation with jumps. 
		For the second derivative, we have
		\begin{align*}
			\frac{d^2}{dx^2} \widetilde{u}(\lambda, x) 
			= (\lambda + r) (I_1+I_2+I_3) - f(x).
		\end{align*}
		The Laplace transform of the first time derivative of $u$ is
		\begin{align*}
			\int_{0}^{\infty} e^{-\lambda t} \frac{d}{dt} u(t,x) dt&= \lambda \widetilde{u}(\lambda, x) - f(x)\\
			&= \lambda (I_1+I_2+I_3) - f(x) + r \int_{0}^\infty f(y) \frac{e^{-y \sqrt{\lambda + r}}}{\sqrt{\lambda + r}} dy.
		\end{align*}
		For the two terms related to the resetting, we get
		\begin{align*}
			- r \widetilde{u}(\lambda, x) = - r( I_1+I_2+I_3) - \frac{r^2}{\lambda} \int_{0}^\infty f(y) \frac{e^{-y \sqrt{\lambda + r}}}{\sqrt{\lambda + r}} dy
		\end{align*}
		and
		\begin{align*}
			r \widetilde{u}(\lambda, 0) 
			= r \int_{0}^\infty f(y) \frac{e^{-y \sqrt{\lambda + r}}}{\sqrt{\lambda + r}} dy + \frac{r^2}{\lambda} \int_{0}^\infty f(y) \frac{e^{-y \sqrt{\lambda + r}}}{\sqrt{\lambda + r}} dy.
		\end{align*}
		By putting everything together, we verify that
		\begin{align*}
			\lambda  \widetilde{u}(\lambda, x) - f(x) = \frac{d^2}{dx^2} \widetilde{u}(\lambda, x) - r \widetilde{u}(\lambda, x) + - r \widetilde{u}(\lambda, 0),
		\end{align*}
		and, by the uniqueness of the Laplace transform inversion (Lerch's theorem), we conclude that \( u(t,x)=\mathbf{E}_x[f (X_t^+)] \) is the solution. 
		Since \(f \in C_b[0,\infty)\), then we have \( \vert \vert u \vert \vert_\infty \leq \vert \vert f \vert \vert_\infty\).
	\end{proof}
	\begin{proof}[Proof of Proposition \ref{prop:stationary}]
		Let \(t_0,..., t_n\) be a sequence of times such that \(0=t_0 < t_1 <...<t_n=T\). For the finite dimensional distributions of \((X^R_t , t \in [0,T], \mathbf{P}_{\mu^+})\)
		\begin{align*}
			&\mathbf{E}_{\mu^+}[f(X^R_{t_0},..., X^R_{t_n})]=\mathbf{E}_{\mu^+}[f(X^+_{T-t_0},..., X^+_{T-t_n})]\\
			&= \int \mu^+(x_n) dx_n \, \mathbf{P}_{x_n}(X^+_{t_n - t_{n-1}} \in dx_{n-1}) \,
			\mathbf{P}_{x_{n-1}}(X^+_{t_{n-1} - t_{n-2}} \in dx_{n-2}) \cdots \\
			&\quad \cdots \mathbf{P}_{x_2}(X^+_{t_2 - t_1} \in dx_1) \,
			\mathbf{P}_{x_1}(X^+_{t_1} \in dx_0) \, f(x_0, \ldots, x_n) \\
			&= \int \mu^+(x_0) dx_0 \, \frac{1}{\mu^+(x_0)} \mathbf{P}_{x_1}(X^+_{t_1} \in dx_0) \mu^+(x_1) \, \frac{1}{\mu^+(x_1)} \mathbf{P}_{x_2}(X^+_{t_2 - t_1} \in dx_1) \mu^+(x_2) \cdots \\
			&\quad \cdots  \mu^+(x_{n-1}) \, \frac{1}{\mu^+(x_{n-1})} \mathbf{P}_{x_n}(X^+_{t_n - t_{n-1}} \in dx_{n-1}) \mu^+(x_n) dx_n \, f(x_0, \ldots, x_n),
		\end{align*}
		where we have used that \(\{x : \mu^+(dx)=0\} \) has zero Lebesgue measure. But, for definition 
		\begin{align*}
			\mathbf{E}_{\mu^+}[f(X^R_{t_0},..., X^R_{t_n})]&= \int \mu^+(x_0) dx_0 \, 
			\mathbf{P}_{x_0}(X^R_{t_1} \in dx_1) 
			\mathbf{P}_{x_1}(X^R_{t_2 - t_1} \in dx_2) \cdots \\
			&\cdots  
			\mathbf{P}_{x_{n-1}}(X^R_{t_n - t_{n-1}} \in dx_n) 
			f(x_0, \ldots, x_n),
		\end{align*}
		then \eqref{duality-density} holds. For the semigroups, we observe that
		\begin{align*}
			\langle P^+_t f , g \rangle_{\mu^+} &= \int_0^\infty P^+_t f(x) g(x) \mu^+(x) dx\\
			&=\int_0^\infty g(x) \left(\int_0^\infty \mathbf{P}_x (X^+_t \in dy) f(y) \right) \mu^+(x) dx\\
			&=\int_0^\infty g(x) \left(\int_0^\infty \frac{\mathbf{P}_y (X^R_t \in dx)}{\mu^+(x)} f(y) \right) \mu^+(x) \mu^+(y) dy\\
			&= \int_0^\infty f(y) \mu^+(y)\left( \int_0^\infty \mathbf{P}_y(X^R_t \in dx) g(x) \right) dy\\
			&=\int_0^\infty f(y) P^R_t g(y) \mu^+(y) dy\\
			&=\langle  f , P^R_t g \rangle_{\mu^+},
		\end{align*}
		where we have used \eqref{duality-density}. The duality formula for the infinitesimal generators \eqref{duality-generator} follows from standard arguments, by taking the limit of the corresponding formula for the semigroups.
	\end{proof}
	\begin{proof}[Proof of Theorem \ref{thm:time-reverse}]
		First, we observe that \((\widetilde{A}, {D}(\widetilde{A}))\) satisfies \eqref{duality-generator}. Indeed, for \(g \in {D}(\widetilde{A})\) and \(f \in {D}(A^+)\), we see
		\begin{align*}
			\int_0^\infty f(x) \widetilde{A}g(x) \mu^+(dx) &= \int_0^\infty f(x) [g''(x) - 2 \sqrt{r} g'(x)] \mu^+(x)dx\\
			&=\int_0^\infty f(x) g''(x) \mu^+(x) dx - 2 \sqrt{r} \int_0^\infty f(x) g'(x) \mu^+(x) dx\\
			&=: I_1 + I_2.
		\end{align*}
		We have
		\begin{align*}
			&I_1= \int_0^\infty f(x) g''(x) \mu^+(x) dx = [f(x) g'(x) \mu^+(x)]^\infty_0-\int_0^\infty g'(x) \frac{d}{dx} [f(x) \mu^+(x)] dx\\
			&=-\sqrt{r} f(0) g'(0) - \int_{0}^{\infty} g'(x) f'(x) \mu^+(x) dx + \sqrt{r} \int_0^\infty g'(x) f(x) \mu^+(x) dx\\
			&=r f(0) \int_0^\infty (g(x) - g(0))\mu^+(x) dx - [g(x) f'(x) \mu^+(x)]^\infty_0 + \int_0^\infty g(x) \frac{d}{dx} [f'(x)\mu^+(x)] dx\\
			& \, + [\sqrt{r} g(x) f(x)\mu^+(x)]^\infty_0- \sqrt{r} \int_0^\infty g(x) \frac{d}{dx} [f(x) \mu^+(x)] dx\\
			&=r \int_0^\infty f(0) g(x) \mu^+(x) dx - r f(0) g(0) + \int_0^\infty g(x) f''(x) \mu^+(x) dx - \sqrt{r} \int_0^\infty g(x) f'(x) \mu^+(x) dx\\
			& \ - rf(0) g(0) - \sqrt{r}\int_0^\infty g(x) f'(x) \mu^+(x) dx + r \int_0^\infty g(x) f(x) \mu^+(x) dx\\
			&=\int_0^\infty f''(x) g(x) \mu^+(x) dx + r \int_0^\infty f(0) g(x) \mu^+(x) dx - 2 r f(0) g(0)\\
			&+ r \int_0^\infty f(x) g(x) \mu^+(x) dx - 2 \sqrt{r} \int_0^\infty f'(x) g(x) \mu^+(x) dx.
		\end{align*}
		For $I_2$, we get
		\begin{align*}
			I_2&=-2 \sqrt{r} \int_0^\infty f(x) g'(x) \mu^+(x) dx\\
			&=-2 \sqrt{r} [f(x) g(x) \mu^+(x)]^\infty_0 + 2 \sqrt{r}\int_0^\infty g(x) \frac{d}{dx}[f(x) \mu^+(x)] dx\\
			&= 2 r f(0) g(0) + 2 \sqrt{r} \int_0^\infty g(x) f'(x) \mu^+(x) dx - 2r\int_0^\infty f(x) g(x) \mu^+(x) dx.
		\end{align*}
		Then, we conclude
		\begin{align*}
			I_1+I_2&= \int_0^\infty f''(x) g(x) \mu^+(x) dx + r \int_0^\infty f(0) g(x) \mu^+(x) dx - r \int_0^\infty f(x) g(x) \mu^+(x) dx\\
			&=\int_0^\infty A^+ f(x) g(x) \mu^+(x) dx
		\end{align*}
		and the first claim holds. 
		Now we show that
		\begin{align*}
			\mathbf{E}_{\mu^+} [f(X^R_t)]=\mathbf{E}_{\mu^+}[f(\widetilde{X}_t)].
		\end{align*}
		for \(f \in C_b(0,\infty)\) and \(t \in [0,T]\). 
		Since \(\mu^+\) is the stationary distribution for \(X^+\) and \(T-t \geq 0\), then we have
		\begin{align*}
			\mathbf{E}_{\mu^+} [f(X^R_t)]=	\mathbf{E}_{\mu^+}[f(X^+_{T-t})]=\int_0^\infty P_{T-t}^+ f(x) \mu^+(dx)= \int_0^\infty f(x) \mu^+(dx).
		\end{align*}
		First, we see that \(\mu^+\) is also a stationary measure for the process \(\widetilde{X}\), which is equivalent to saying that, \cite[Proposition 9.2]{markov-book},
		\begin{align*}
			\int_0^\infty \widetilde{A}f(x) \mu^+(dx)=0
		\end{align*}
		with \(f^\prime(x) + \mathbf{D}^\Psi_x f(x)=0\) in \(x=0\). We rewrite the boundary condition as
		\begin{align}
			\label{bc-reverse}
			\left(f^\prime(x) + \mathbf{D}^\Psi_x f(x)\right)\vert_{x=0}&=f^\prime(0) + \int_0^\infty (f(x) - f(0)) r e^{-\sqrt{r} x} dx \notag\\
			&=f^\prime(0) - \sqrt{r} f(0) + r \int_0^\infty f(x) e^{-\sqrt{r} x} dx=0.
		\end{align}
		Then, for \(f \in {D(\widetilde{A})} \)
		\begin{align*}
			\int_0^\infty \widetilde{A}f(x) \mu^+(dx)&= \int_0^\infty(f^{\prime \prime}(x) - 2 \sqrt{r} f^\prime(x)) \sqrt{r} e^{-\sqrt{r} x} dx\\
			&=[\sqrt{r} f^\prime(x) e^{-\sqrt{r} x}]^\infty_0 + r \int_0^\infty f^\prime(x) e^{-\sqrt{r}x} dx - 2r\int_0^\infty f^\prime(x) e^{-\sqrt{r}x} dx\\
			&=-\sqrt{r} f^\prime(0) - r \int_0^\infty f^\prime(x) e^{-\sqrt{r}x} dx\\
			&=-\sqrt{r} f^\prime(0) - [r f(x) e^{-\sqrt{r} x}]^\infty_0 - r\sqrt{r} \int_0^\infty f(x) e^{-\sqrt{r}x} dx\\
			&= -\sqrt{r} f^\prime(0) + r f(0) - r\sqrt{r} \int_0^\infty f(x) e^{-\sqrt{r}x} dx\\
			&= -\sqrt{r}(f^\prime(0) - \sqrt{r} f(0) + r \int_0^\infty f(x) e^{-\sqrt{r} x} dx)=0,
		\end{align*}
		where in the last equality we have used \eqref{bc-reverse}. 
		Thus, denoting by \(\widetilde{P}\) the semigroup of \(\widetilde{X}\) and using the fact that \(\mu^+\) is also a stationary measure for \(\widetilde{X}\), we obtain that 
		\begin{align*}
			\mathbf{E}_{\mu^+}[f(\widetilde{X}_t)]=\int_0^\infty \widetilde{P}_tf(x) \mu^+(dx)=\int_0^\infty f(x) \mu^+(dx)=\mathbf{E}_{\mu^+}[f(X^R_{t})],
		\end{align*}
		and our claim holds.
	\end{proof}
	\begin{proof}[Proof of Theorem \ref{thm:paths}]
		First, we observe what happens in a fixed time interval \([0,k]\), for \(k>0\) and \(k \ne 1\). Since \(B^+_0=0\), the density is \(\mathbf{P}_0 (B^+_t \in dx) = 2 g(t,x)\).  
		From \cite[Theorem 2.1, Formula (2.4)]{reverse-diffusion1} and \cite[Lemma 2.1, Formula (2.8)]{reverse-reflected}, we have that the only effect of choosing the interval \([0,1]\) is in the change of variable \(1-t\). Therefore, we can obtain the same results, i.e. the same generator \eqref{Lsegnato+}, by considering \(\sqrt{2} W^+\) on \([0,k]\), with the change of variable \(k-t\), and we have that \(B^+_{k-t}\) is governed by
		\begin{align*}
			\overline{L}^k_t f(x)&= \frac{d^2}{dx^2} f(x) + \left(\frac{2}{2 g(k-t,x)} \frac{d}{dx} 2 g(k-t,x)\right) \frac{d}{dx} f(x)\\
			&= \frac{d^2}{dx^2} f(x) + \left(\frac{2}{ g(k-t,x)} \frac{d}{dx}  g(k-t,x)\right) \frac{d}{dx} f(x)
		\end{align*}
		with \(\frac{d}{dx} f(x)=0\).
		We now consider the case where \( k \) is random. The difference when computing the change of variable \( T_1 - t \) lies in the fact that \( T_1 \) now follows an \(\text{Exp}(r)\) distribution, as it represents the first jump time of \( N \), and therefore
		\begin{align*}
			2 g(T_1-t,x) dx&=\int_0^\infty \mathbf{P}_0(B^+_{k-t} \in dx) \mathbf{P}(T_1 \in dk)\\&=\int_0^\infty 2 g(k-t,x) r e^{-rk} dk dx\\
			&=2 \sqrt{r} e^{-rt} e^{-x\sqrt{r}} dx.
		\end{align*}
		From the new density under the time change \( T_1 - t \), we obtain that the drift term is  
		\[\frac{2}{ g(T_1-t,x)} \frac{d}{dx}  g(T_1-t,x) =  \frac{2}{\sqrt{r} e^{-rt} e^{-x\sqrt{r}}} \frac{d}{dx} \sqrt{r} e^{-rt} e^{-x\sqrt{r}}=-2\sqrt{r}. \]
		Thus, we conclude that the generator of \( B^+_{T_1 - t} \) is  
		\[	\widetilde{A}^+ f(x)= \frac{d^2}{dx^2} f(x) - 2 \sqrt{r} \frac{d}{dx} f(x)
		\]
		with the boundary condition \(\frac{d}{dx} f(0)=0\), since it is not affected by the time change, and it coincides with the one of the process \(\widetilde{B}\).
	\end{proof}
	\begin{proof}[Proof of Theorem \ref{thm:trace}]
		First, we see that
		\begin{align*}
			T_1 := B_1 \circ (L^\Phi)^{-1} \, ; \quad \text{and }\quad T_2:= B_2 \circ (L^\Psi \circ \widetilde{\gamma})^{-1},
		\end{align*}
		where we denote with $-1$  the inverse process and we have used Theorem \ref{tm:inverselocaltime} and Theorem \ref{tm:local-time-pallino} for the characterization of \(\ell^+\) and \(\widetilde{\ell}\) as inverses of subordinators.
		But, from Theorem \ref{tm:same-local-times}, we know that $L^\Phi$ and $L^\Psi \circ \widetilde{\gamma}$ have the same law, than also $T_1$ and $T_2$ have the same law. In particular, we know that the inverse of \(L^\Phi\) is the subordinator \(H^\Phi\), then \(T_1\) and \(T_2\) are subordinated Brownian motions, and their generator when is restricted to ${D}(\Delta)$, is given by the Phillips' representation (or the Bochner subordination) \cite[Theorem 32.1]{sato}
		\begin{align*}
			-\Phi(-\Delta) v(x)= \int_0^\infty (S_zv(x) - v(x)) \Pi^\Phi (dz),
		\end{align*}
		where $\Pi^\Phi(dz)= \frac{e^{-rz} (2rz+1)}{2 \sqrt{\pi} z^{3/2}}\, dz$ is the L\'evy measure related to $\Phi$ and $S$ is the semigroup associated to the Brownian motion on $\mathbb{R}$, such that the characteristic symbol is $\widehat{S}_z = e^{-z \xi^2}$. Thus, we have that the Fourier transform is, for $\xi \in \mathbb{R}$,
		\begin{align}
			\label{fourier-trace}
			\int_{\mathbb{R}} e^{-i \xi x } \left(-\Phi(-\Delta) v(x)\right) dx= -\Phi(\xi^2) \widehat{v}(\xi)= -\frac{\xi^2}{\sqrt{\xi^2 + r}} \widehat{v}(\xi).
		\end{align}
		where $\widehat{v}$ is the Fourier transform of the function $v$. Now, we observe how the operators $K_1$ and $K_2$ share the same Fourier transform as in \eqref{fourier-trace}. The problem \eqref{P1} on \( H \) can be rewritten as
		\begin{align*}
			\frac{\partial^2}{\partial x^2} u_1(x,y) + \frac{\partial^2}{\partial y^2} u_1(x,y) + r (u_1(x,0) - u_1(x,y))=0,
		\end{align*}
		by applying the Fourier transform $\mathcal{F}_x$ on the first variable, we obtain
		\begin{align*}
			- \xi^2 \mathcal{F}_x u_1(\xi,y) + \frac{\partial^2}{\partial y^2} \mathcal{F}_x u_1(\xi,y) + r (\mathcal{F}_x u_1(\xi,0) - \mathcal{F}_x u_1(\xi,y))=0.
		\end{align*}
		Then, the bounded solution of this ODE is
		\begin{align}
			\label{ODE-resetting}
			\mathcal{F}_x u_1(\xi,y) = \frac{\xi^2}{\xi^2 + r} e^{-y\sqrt{\xi^2 + r} } \widehat{f}(\xi) + \frac{r}{\xi^2+r} \widehat{f}(\xi),
		\end{align}
		indeed 
		\begin{align*}
			\mathcal{F}_x u_1(\xi,0)= \widehat{f}(\xi)
		\end{align*}
		and, by replacing in the ODE, we conclude
		\begin{align*}
			- (\xi^2 + r) \left(\frac{\xi^2}{\xi^2 + r} e^{-y\sqrt{\xi^2 + r} } \widehat{f}(\xi) + \frac{r}{\xi^2+r} \widehat{f}(\xi)\right) + (\xi^2+r) \frac{\xi^2}{\xi^2 + r} e^{-y\sqrt{\xi^2 + r} } + r \widehat{f}(\xi)=0.
		\end{align*}
		Thus, the Fourier transform of the Dirichlet-to-Neumann operator $K_1$ is
		\begin{align*}
			\mathcal{F} K_1 f(\xi)= \frac{\partial}{\partial y} \mathcal{F}_x u_1(\xi,y) \big\vert_{y=0}= -\frac{\xi^2 \sqrt{\xi^2 + r}}{\xi^2 + r} \widehat{f}(\xi)=- \frac{\xi^2}{\sqrt{\xi^2 + r}} \widehat{f}(\xi),
		\end{align*}
		hence it coincides with \eqref{fourier-trace} as we expected by the definition as subordinate Brownian motion. Now, we move on \eqref{P2} and we rewrite the problem on \(H\) as
		\begin{align*}
			\frac{\partial^2}{\partial x^2} u_2(x,y) + \frac{\partial^2}{\partial y^2} u_2(x,y) - 2 \sqrt{r} \frac{\partial}{\partial y} u_2(x,y)=0.
		\end{align*}
		By applying the Fourier transform on the first variable, we obtain
		\begin{align*}
			- \xi^2 \mathcal{F}_x u_2(\xi, y) + \frac{\partial^2}{\partial y^2} u_2(\xi,y) - 2 \sqrt{r} \frac{\partial}{\partial y} u_2(\xi,y)=0.
		\end{align*}
		The general solution of this ODE is
		\begin{align*}
			\mathcal{F}_x u_2(\xi, y)= (\alpha e^{-\frac{y}{2}(\sqrt{4 \xi^2 + 4r} - 2\sqrt{r})} + \beta e^{\frac{y}{2}(\sqrt{4 \xi^2 + 4r} + 2\sqrt{r})}) \widehat{f}(\xi),
		\end{align*}
		since $\sqrt{4 \xi^2 + 4r} - 2\sqrt{r}>0$ and we are dealing with bounded solutions, we have $\alpha=1, \, \beta=0$. Then, we have
		\begin{align*}
			\mathcal{F}_x u_2(\xi,y)= e^{-y (\sqrt{\xi^2 +r} - \sqrt{r})} \widehat{f}(\xi). 
		\end{align*}
		For the operator $K_2$ we get
		\begin{align*}
			\mathcal{F} K_2 f(\xi)&= \frac{\partial}{\partial y} u_2(\xi, y) \big\vert_{y=0} + r \int_0^\infty (\mathcal{F}_x u_2(\xi,y) - \mathcal{F}_x u_2(\xi,0)) e^{-\sqrt{r} y} dy\\
			&=(-\sqrt{\xi^2 +r} + \sqrt{r}) \widehat{f}(\xi) + r \int_0^\infty (e^{-y((\sqrt{\xi^2 +r} - \sqrt{r}))} - 1) e^{-\sqrt{r} y} dy\\
			&=(-\sqrt{\xi^2 +r} + \sqrt{r}) \widehat{f}(\xi) + r \widehat{f}(\xi) \left(\frac{1}{\sqrt{\xi^2 + r} - \frac{1}{\sqrt{r}}}\right)\\
			&=- \frac{\xi^2}{\sqrt{\xi^2 + r}} \widehat{f}(\xi),
		\end{align*}
		as we expected. Thus, \(K_1\) and \(K_2\) share the same Fourier transform, which coincides with the one of the generator of the trace processes. We note that, since we assumed \( f \in \mathcal{S}(\mathbb{R}) \), it follows that \( \widehat{f} \in \mathcal{S}(\mathbb{R}) \) as well. Therefore, \( \mathcal{F} K_1 \) and \( \mathcal{F} K_2 \) also belong to \( \mathcal{S}(\mathbb{R}) \).
		By applying the inverse Fourier transform, we obtain the equality.
\end{proof}
\section*{Appendix}
	\textbf{Background on Subordinators}
	The Laplace exponent \(\Phi\) of the subordinator, in \eqref{LapH}, is uniquely described by the non-negative real number \(d\) and by the L\'evy measure \(\Pi^\Phi\) on \((0, \infty)\), such that
	\[
	\int_0^\infty (1 \wedge z) \Pi^\Phi(dz) < \infty.
	\]
	The L\'evy-Khintchine representation for the Laplace exponent \(\Phi\) is given by (\cite[Theorem 3.2]{schilling2012bernstein})
	\begin{align}
		\label{LevKinFormula}
		\Phi(\lambda) =  d \lambda + \int_0^\infty (1 - e^{-\lambda z}) \Pi^\Phi(dz), \quad \lambda > 0,
	\end{align}
	where the drift coefficient \(d\) is defined as
	\begin{align*}
		d = \lim_{\lambda \to \infty} \frac{\Phi(\lambda)}{\lambda}.
	\end{align*}
	The function \(\Phi\) is a Bernstein function, so such that
	\begin{align*}
		(-1)^{n-1} \Phi^{(n)}(\lambda) \geq 0 \quad \text{ for } n = 1, 2, \ldots
	\end{align*}
	uniquely associated with \(H^\Phi\) (\cite[Theorem 5.2]{schilling2012bernstein}). For the reader's convenience, we also recall that
	\begin{align}
		\label{tailSymb}
		\frac{\Phi(\lambda)}{\lambda} = d + \int_0^\infty e^{-\lambda z} \overline{\Pi}^\Phi(z) \, dz, \quad \overline{\Pi}^\Phi(z) =  \Pi^\Phi(z, \infty),
	\end{align}
	where \(\overline{\Pi}^\Phi\) denotes the \textit{tail} of the L\'evy measure \(\Pi^\Phi\).\\
	For completeness, let us also recall that \cite[Formula (3.13)]{Meerschaert-triangular}
	\begin{align}
		\label{symbol:L}
		\int_{0}^\infty e^{-\lambda t} \mathbf{P}_0 (L_t^\Phi \in dx) dt = \frac{\Phi(\lambda)}{\lambda} e^{-x\Phi(\lambda)}.
	\end{align}

We now provide the results we need on reflecting Brownian motion with drift and its local time.
\begin{theorem}
	For the reflected Brownian motion, with drift $-2\sqrt{r}$, $\widetilde{B}$ we have
	\begin{align}
		\label{drift:density}
		\mathbf{P}_x(\widetilde{B}_t \in dy)= e^{-rt} e^{\sqrt{r}(x-y)} [g(t, y-x) + g(t, y+x) + 2 \sqrt{r} \int_0^\infty e^{\sqrt{r} w} g(t, w+x+y) dw] \, dy
	\end{align}
	and the Laplace transform of the joint density with the local time at zero is
	\begin{align}
		\label{drift:joint}
		\int_0^\infty e^{-\lambda t} \mathbf{P}_0(\widetilde{B}_t \in dy, \widetilde{\gamma}_t \in dw) dt = e^{-\sqrt{r} (y-w)} e^{- \sqrt{\lambda + r}(y+w) } dy \, dw.
	\end{align}
\end{theorem}
\begin{proof}
	The formula \eqref{drift:density} is a straightforward adaptation of the density found in \cite[Appendix 1, Section 16]{BorodinSalminen}.
	For the formula \eqref{drift:joint}, we simply apply the time Laplace transform for \cite[Formula 1.3.8, page 254]{BorodinSalminen} for \(x=0\).
\end{proof}
\begin{proof}[Proof of Theorem \ref{NLBVP}]
	\phantomsection
	\label{proof:NLBVP}
	We want to provide that \(u(t,x)=\mathbf{E}_x[f (\widetilde{X}_t)]\) is the solution of the problem. We move to the resolvent, for \(\lambda>0\),
	\begin{align*}
		R_\lambda f(x)&=\int_0^\infty e^{-\lambda t} u(t,x) dt \\
		&=\mathbf{E}_x \left[\int_0^\infty e^{-\lambda t} f\left(\widetilde{X}_t\right)\right].
	\end{align*}
	We introduce
	\begin{align*}
		\tau_0:=\inf\{s \geq 0 : \widetilde{X}_s=0\}
	\end{align*}
	and we observe that, before $\tau_0$, the process $\widetilde{X}$ behaves as $\widetilde{B}^D$ a Brownian motion with drift $-2\sqrt{r}$ killed at zero related to Dirichlet boundary conditions. We also recall, \cite[Formula 2.0.1, page 295]{BorodinSalminen}
	\begin{align*}
		\mathbf{E}_x\left[e^{-\lambda \tau_0}\right]=e^{-x(\sqrt{\lambda+ r} - \sqrt{r})}.
	\end{align*}
	Then, we rewrite the resolvent as, $\lambda>0$,
	\begin{align}
		\label{resolvent}
		R_\lambda f(x) &= \mathbf{E}_x\left[ \int_0^\infty e^{-\lambda t} f \left(\widetilde{X}_t \right) dt \right] \notag \\
		&=\mathbf{E}_x\left[ \int_0^{\tau_0} e^{-\lambda t} f \left(\widetilde{X}_t \right) dt \right] + \mathbf{E}_x\left[ \int_{\tau_0}^\infty e^{-\lambda t} f \left(\widetilde{X}_t \right) dt \right] \notag \\
		&=\mathbf{E}_x\left[ \int_0^\infty e^{-\lambda t} f \left(\widetilde{B}^D_t \right) dt \right] + \mathbf{E}_x\left[e^{-\lambda \tau_0}\right] \mathbf{E}_0\left[ \int_0^\infty e^{-\lambda t} f \left(\widetilde{X}_t \right) dt \right] \notag \\
		&= R_\lambda^D f(x) + e^{-x(\sqrt{\lambda+ r} - \sqrt{r})} R_\lambda f(0),
	\end{align}
	where $R_\lambda^D$ is the resolvent associated to $\widetilde{B}^D_t$ and its density is
	\begin{align*}
		\mathbf{P}_x(\widetilde{B}^D_t \in dy)/dy = e^{- rt} e^{\sqrt{r}(x-y)} \left[g(t,y-x) - g(t,y+x)\right],
	\end{align*}
	where  $g(t, z) = e^{-z^2/ 4t}/\sqrt{4 \pi t}$. 
	We recall the well-known transforms
	\begin{align}
		\label{gLap1}
		\int_0^\infty e^{-\lambda t} g(t,x) dt = \frac{1}{2}\frac{e^{-\vert x \vert \sqrt \lambda}}{\sqrt \lambda}, \quad \lambda >0
	\end{align}
	and
	\begin{align}
		\label{gLap2}
		\int_0^\infty e^{-\lambda t} \frac{\vert x \vert }{t} g(t,x) dt = e^{- \vert x \vert \sqrt \lambda}, \quad \lambda>0.
	\end{align}
	Then, from \eqref{gLap1}, we have, $\lambda>0$,
	\begin{align*}
		R_\lambda^D f(x) &= \mathbf{E}_x\left[ \int_0^\infty e^{-\lambda t} f \left(\widetilde{B}^D_t \right) dt \right] \\
		&=\int_0^\infty \int_0^\infty f(y) e^{-\lambda t} e^{- rt} e^{\sqrt{r}(x-y)} \left[g(t,y-x) - g(t,y+x)\right] dt \, dy\\
		&=\frac{1}{2}\int_0^\infty \left(\frac{e^{-\vert x-y \vert \sqrt{\lambda+ r}}}{\sqrt{\lambda+ r}}- \frac{e^{-( x+y)  \sqrt{\lambda+ r}}}{\sqrt{\lambda+ r}} \right) e^{\sqrt{r}(x-y)}\, f(y) dy\\
		&=\frac{1}{2} \int_0^\infty \frac{e^{( x-y)  \sqrt{\lambda+ r}}}{\sqrt{\lambda+ r}} e^{\sqrt{r}(x-y)}\, f(y) dy - \frac{1}{2} \int_0^\infty \frac{e^{-( x+y)  \sqrt{\lambda+ r}}}{\sqrt{\lambda+ r}} e^{\sqrt{r}(x-y)} \, f(y) dy\ +\notag\\
		& - \frac{1}{2} \int_0^x \left( \frac{e^{( x-y)  \sqrt{\lambda+ r}}}{\sqrt{\lambda+ r}} - \frac{e^{-( x-y)  \sqrt{\lambda+ r}}}{\sqrt{\lambda+ r}} \right) e^{\sqrt{r}(x-y)}\, f(y) dy.
	\end{align*}
	We can easily verify that 
	\begin{align*}
		\frac{d^2}{dx^2}R_\lambda f(x) - 2\sqrt{r} \frac{d}{dx} R_\lambda f(x) = \lambda R_\lambda f(x) - f(x), \quad x>0, \lambda > 0,
	\end{align*}
	then the heat equation with drift \(-2 \sqrt{r}\) is verified. We now move on the boundary conditions, but first we have to understand the behaviour of the additive part \(R^\Psi\).

	We know that the process $H^\Psi L^\Psi$  is right-continuous with jumps, since $H^\Psi$ is right-continuous with jumps and $L^\Psi$ is continuous. If we enumerate the jumps of $H^\Psi$ with $l_1, l_2,...$, we can decompose $(0,\infty)$ in two sets, as done in \cite[Lemma 3.5]{Pilipenko}:
	\begin{align*}
		&\mathcal{J}^c:=\{t \geq 0 : H^\Psi L^\Psi_t = t\}\\
		&\mathcal{J} := \bigcup_{n \geq 1} \mathcal{J}_n= \bigcup_{n \geq 1} [l_n^{-}, l_n^+),
	\end{align*}
	where we have to include the point $l_n^-$ because $H^\Psi L^\Psi$ is right-continuous. From \eqref{pallino}, we see that the process outside from the jumps of $H^\Psi$ is a drifted reflecting Brownian motion otherwise we have to consider the jump. In particular, we have
	\begin{align*}
		\widetilde{X}_t=
		\begin{cases}
			\widetilde{B}_t \quad & \gamma_t \in \mathcal{J}^c\\
			l_n^+-\widetilde{\gamma}_t + \widetilde{X}_t \quad & \gamma_t \in \mathcal{J}_n.
		\end{cases}
	\end{align*}
	indeed if $\widetilde{\gamma},_t \in \mathcal{J}_n$ we have $H^\Psi L^\Psi \widetilde{\gamma}= l_n^+$. We observe that $\widetilde{\gamma}_t \in [l_n^-,l_n^+)$ if and only if $t \in [\widetilde{\gamma}_{-}^{-1}(l_n^-), \widetilde{\gamma}_{-}^{-1}(l_n^+))$, in fact $\widetilde{\gamma}_t$ is a continuous process and its right inverse $\widetilde{\gamma}_t^{-1}$ is a right-continuous process with jumps, then to make sure that the point $l_n^-$ is included we have to introduce
	\begin{align*}
		\widetilde{\gamma}_{-}^{-1}(t):=\inf \{s \geq 0 : \widetilde{\gamma}_s \geq t\},
	\end{align*}  
	which is left-continuous. We prefer writing $\widetilde{\gamma}_{-}^{-1}(t)$ over $\widetilde{\gamma}_{t-}^{-1}$, to avoid confusion with right and left points of $l_n$.  Then, the resolvent is
	\begin{align*}
		\mathbf{E}_{0}  \left[\int_0^\infty e^{-\lambda t} f (\widetilde{X}_t) dt\right]&= \mathbf{E}_{0}  \left[\int_{\mathcal{J}\cup \mathcal{J}^c} e^{-\lambda t} f(\widetilde{X}_t) dt\right]
		\\
		&=\mathbf{E}_{0}  \left[\int_{\mathcal{J}} e^{-\lambda t} f(\widetilde{X}_t) dt\right] + \mathbf{E}_{0}  \left[\int_{\mathcal{J}^c} e^{-\lambda t} f(\widetilde{X}_t) dt\right]\\
		&=\sum_{n \geq 1} \mathbf{E}_{0}  \left[\int_{\widetilde{\gamma}_{-}^{-1}(l_n^-)}^{\widetilde{\gamma}_{-}^{-1}(l_n^+)} e^{-\lambda t} f(\widetilde{X}_t) dt\right] + \mathbf{E}_{0}  \left[\int_{\mathcal{J}^c} e^{-\lambda t} f(\widetilde{X}_t) dt\right].
	\end{align*} 
	On \(\mathcal{J}\), we have that
	\begin{align*}
		&\sum_{n \geq 1} \mathbf{E}_{0}  \left[\int_{\widetilde{\gamma}_{-}^{-1}(l_n^-)}^{\widetilde{\gamma}_{-}^{-1}(l_n^+)} e^{-\lambda t} f(\widetilde{X}_t) dt\right]=\\
		&=\sum_{n \geq 1}  \mathbf{E}_{0} \left[\int_0^{\widetilde{\gamma}_{-}^{-1}(l_n^+)-\widetilde{\gamma}_{-}^{-1}(l_n^-)} e^{-\lambda (s + \widetilde{\gamma}_{-}^{-1}(l_n^-))} f(\widetilde{X}_{s + \widetilde{\gamma}_{-}^{-1}(l_n^-)}) ds\right].
	\end{align*}
	Before proceeding with the integration, we need to make some observations about the local time $\widetilde{\gamma}$ and its left-inverse $\widetilde{\gamma}_{-}^{-1}$. Since $\widetilde{\gamma}$ is continuous, we have $\widetilde{\gamma}_{\widetilde{\gamma}^{-1}_{-}(l_n^-)}=l_n^-$. For the left inverse, serving as the upper bound within the integral, we know that (see \cite[Proposition 1.3, Chapter X]{revuz-yor})
	\begin{align*}
		\widetilde{\gamma}^{-1}_{-}(l_n^+) - \widetilde{\gamma}{-}^{-1}(l_n^-)=\widetilde{\gamma}_{-}^{-1}(l_n) \circ \theta _{\widetilde{\gamma}^{-1}_{-}(l_n^-) },
	\end{align*}
	where $l_n=l_n^+ - l_n^-$ and $\theta$ is the shift operator. We also recall the additive functional property for the local time:
	\begin{align*}
		\widetilde{\gamma}_{s+\widetilde{\gamma}^{-1}_{-}(l_n^-)}=\widetilde{\gamma}_s \circ \theta_{\widetilde{\gamma}^{-1}_{-}(l_n^-) } + \widetilde{\gamma}_{\widetilde{\gamma}^{-1}_{-}(l_n^-)}= \widetilde{\gamma}_s \circ \theta_{\widetilde{\gamma}^{-1}_{-}(l_n^-) } + l_n^-.
	\end{align*}
	Then, the resolvent on \(\mathcal{J}\) is rewritten as
	\begin{align*}
		&\sum_{n \geq 1}  \mathbf{E}_{0} \left[\int_0^{\widetilde{\gamma}_{-}^{-1}(l_n^+)-\widetilde{\gamma}_{-}^{-1}(l_n^-)} e^{-\lambda (s + \widetilde{\gamma}_{-}^{-1}(l_n^-))} f(\widetilde{X}_{s + \widetilde{\gamma}_{-}^{-1}(l_n^-)}) ds\right]\\
		&=\sum_{n \geq 1} \mathbf{E}_{0}\left[ \mathbf{E}_{0} \left[ \left(\int_0^{\widetilde{\gamma}_{-}^{-1}(l_n)} e^{-\lambda s } f(l_n + \widetilde{B}_s - \widetilde{\gamma}_s) ds\right) \circ \theta_{\widetilde{\gamma}^{-1}_{-}(l_n^-) }\Big\vert \mathcal{F}_{\theta_{\widetilde{\gamma}_{-}^{-1}(l_n^-)}}\right]\right]
	\end{align*}
	where $\mathcal{F}$ is the natural filtration of $\widetilde{B}$. By using the strong Markov property of $\widetilde{B}$ with respect to the stopping time $\widetilde{\gamma}_{-}^{-1}(l_n^-)$, since from Fubini's theorem $l_n$ can be seen as a constant for \(\widetilde{B}\), and the fact that $\widetilde{\gamma}_{-}^{-1}(l_n^-)= \widetilde{\gamma}^{-1}_{l_n^-}$ a.s., we obtain
	\begin{align*}
		&\sum_{n \geq 1} \mathbf{E}_{0}\left[ \mathbf{E}_{0} \left[ \left(\int_0^{\widetilde{\gamma}_{-}^{-1}(l_n)} e^{-\lambda s } f(l_n + \widetilde{B}_s - \widetilde{\gamma}_s) ds\right) \circ \theta_{\widetilde{\gamma}^{-1}_{-}(l_n^-) }\Big\vert \mathcal{F}_{\theta_{\widetilde{\gamma}_{-}^{-1}(l_n^-)}}\right]\right]\\
		&=\sum_{n \geq 1} \mathbf{E}_{0}\left[e^{-\lambda \widetilde{\gamma}^{-1}_{l_n^-}} \mathbf{E}_{0} \left[\int_0^{\widetilde{\gamma}_{-}^{-1}(l_n)} e^{-\lambda s } f(l_n +\widetilde{B}_s - \widetilde{\gamma}_s) ds \right] \right]
	\end{align*}
	but, the process  $l_n-\widetilde{\gamma}_t + \widetilde{B}$,  since $t \leq \widetilde{\gamma}^{-1}(l_n) $,  behaves like the drifted Brownian motion \(\widetilde{B}_t\) , for \(t \leq \tau_0\) started at $l_n$ , with $\tau_0$ hitting time at zero for $\widetilde{B}$. Then, for the jump times \(t_n\) of the subordinator, we have
	\begin{align*}
		&\sum_{n \geq 1}  \mathbf{E}_{0} \left[e^{-\lambda \widetilde{\gamma}^{-1}_{l_n^-}} \mathbf{E}_{l_n}\left[ \int_{0}^{\tau_0} e^{-\lambda t} f(\widetilde{B}_t) dt \right]\right]
		=\sum_{n \geq 1} \mathbf{E}_{0} \left[e^{-(\sqrt{\lambda+r}-\sqrt{r}) t_n^-} {R}_\lambda^D f(l_n)\right]\\
		&=\sum_{n \geq 1} \mathbf{E}_{0} \left[e^{-(\sqrt{\lambda+r}-\sqrt{r}) H^\Psi_{l_n^-}} {R}_\lambda^D f(l_n)\right]\\
		&=\mathbf{E}_{0} \left[\int_{(0,\infty) \times (0,\infty)} e^{-(\sqrt{\lambda+r}-\sqrt{r}) H^\Psi_{t^-}} {R}_\lambda^D f(l) N(dt \times dl)\right],
	\end{align*}
	where $N(dt \times dl)$ is the random measure associated to $H^\Psi$. From \cite[Example II.4.1]{ikeda2014stochastic} we know that $N(dt \times dl)= dt \, \Pi^\Psi(dl)$, then we get
	\begin{align*}
		&\mathbf{E}_{0} \left[\int_{(0,\infty) \times (0,\infty)} e^{-(\sqrt{\lambda+r}-\sqrt{r}) H^\Psi_{t^-}} {R}_\lambda^D f( l) N(dt \times dl)\right]\\
		&=\int_0^\infty \int_0^\infty e^{-t \Psi(\sqrt{\lambda+r}-\sqrt{r})} {R}_\lambda^D f( l) dt \,\Pi^\Psi(dl)\\
		&=\frac{\int_0^\infty {R}_\lambda^D f(y) \Pi^\Psi(dl)}{\Psi(\sqrt{\lambda+r}-\sqrt{r})},
	\end{align*}
	which concludes
	\begin{align}
		\label{resolvent:jumps}
		\mathbf{E}_{0}  \left[\int_{\mathcal{J}} e^{-\lambda t} f(\widetilde{X}_t) dt\right]=\frac{\int_0^\infty {R}_\lambda^D f(y) \Pi^\Psi(dl)}{\Psi(\sqrt{\lambda+r}-\sqrt{r})}.
	\end{align}
	Now, for the integral on \(\mathcal{J}^c\), we observe that 
	\begin{align*}
		\mathbf{E}_{0}  \left[\int_{\mathcal{J}^c} e^{-\lambda t} f(\widetilde{X}_t) dt\right]&=\mathbf{E}_{0}  \left[\int_{\mathcal{J}^c} e^{-\lambda t} f(\widetilde{B}_t) dt\right] \\
		&=\mathbf{E}_{0}  \left[\int_0^\infty e^{-\lambda t} f(\widetilde{B}_t) dt\right]-\sum_{n \geq 1} \mathbf{E}_{0}  \left[\int_{\widetilde{\gamma}_{-}^{-1}(l_n^-)}^{\widetilde{\gamma}_{-}^{-1}(l_n^+)} e^{-\lambda t} f(\widetilde{B}_t) dt\right].
	\end{align*}
	For the first integral, by integrating \eqref{drift:joint} in \(dw\), we have
	\begin{align*}
		\mathbf{E}_{0}  \left[\int_0^\infty e^{-\lambda t} f(\widetilde{B}_t) dt\right]= \frac{1 }{\sqrt{\lambda + r} - \sqrt{r}} \int_0^\infty e^{-y(\sqrt{\lambda + r} + \sqrt{r}) } f(y) dy.
	\end{align*}
	For the second integral, we use the same calculation done on \(\mathcal{J}\), and we get
	\begin{align*}
		\sum_{n \geq 1} \mathbf{E}_{0}  \left[\int_{\widetilde{\gamma}_{-}^{-1}(l_n^-)}^{\widetilde{\gamma}_{-}^{-1}(l_n^+)} e^{-\lambda t} f(\widetilde{B}_t) dt\right]=\sum_{n \geq 1} \mathbf{E}_{0}\left[e^{-\lambda \widetilde{\gamma}^{-1}_{l_n^-}} \mathbf{E}_{0} \left[\int_0^{\widetilde{\gamma}_{-}^{-1}(l_n)} e^{-\lambda s } f(\widetilde{B}_s) ds \right] \right], 
	\end{align*}
	and from Dynkin's formula for resolvents at the stopping time \(\widetilde{\gamma}_{-}^{-1}\), we write
	\begin{align*}
		&\sum_{n \geq 1} \mathbf{E}_{0}\left[e^{-\lambda \widetilde{\gamma}^{-1}_{l_n^-}} \mathbf{E}_{0} \left[\int_0^{\widetilde{\gamma}_{-}^{-1}(l_n)} e^{-\lambda s } f(\widetilde{B}_s) ds \right] \right]\\
		&= \sum_{n \geq 1} \mathbf{E}_{0}\left[e^{-\lambda \widetilde{\gamma}^{-1}_{l_n^-}} \left(\mathbf{E}_{0}  \left[\int_0^\infty e^{-\lambda t} f(\widetilde{B}_t) dt\right] (1-e^{-(\sqrt{\lambda + r} - \sqrt{r}) l_n}) \right) \right]\\
		&+ \mathbf{E}_{0} \left[\int_{(0,\infty) \times (0,\infty)} e^{-(\sqrt{\lambda+r}-\sqrt{r}) H^\Psi_{t^-}} (1-e^{-(\sqrt{\lambda + r} - \sqrt{r}) l})  N(dt \times dl)\right] \\
		&\quad \times \left(\mathbf{E}_{0}  \left[\int_0^\infty e^{-\lambda t} f(\widetilde{B}_t) dt\right] \right)\\
		&=\frac{\int_0^\infty(1-e^{-(\sqrt{\lambda + r} - \sqrt{r}) l}) \Pi^\Psi(dl)}{\Psi(\sqrt{\lambda + r} - \sqrt{r})} \, \frac{1 }{\sqrt{\lambda + r} - \sqrt{r}} \int_0^\infty e^{-y(\sqrt{\lambda + r} + \sqrt{r})}f(y) dy.
	\end{align*}
	Then, we conclude
	\begin{align}
		\label{resolvent:no-jumps}
		&\mathbf{E}_{0}  \left[\int_{\mathcal{J}^c} e^{-\lambda t} f(\widetilde{X}_t) dt\right]=\\
		&=\left(1- \frac{\int_0^\infty(1-e^{-(\sqrt{\lambda + r} - \sqrt{r}) l}) \Pi^\Psi(dl)}{\Psi(\sqrt{\lambda + r})} \right)  \frac{1 }{\sqrt{\lambda + r} - \sqrt{r}}  \int_0^\infty e^{-y(\sqrt{\lambda + r} + \sqrt{r})} f(y) dy \notag\\
		&=\frac{1}{\Psi(\sqrt{\lambda + r} - \sqrt{r})} \int_0^\infty e^{-y(\sqrt{\lambda + r} + \sqrt{r})} f(y) dy.
	\end{align}
	By collecting \eqref{resolvent:jumps} and \eqref{resolvent:no-jumps}, the resolvent at the boundary is
	\begin{align}
		\label{resolvent:Xtilde0}
		R_\lambda f(0)&= \mathbf{E}_{0}  \left[\int_0^\infty e^{-\lambda t} f (\widetilde{X}_t) dt\right] \notag \\
		&=\mathbf{E}_{0}  \left[\int_{\mathcal{J}} e^{-\lambda t} f(\widetilde{X}_t) dt\right] + \mathbf{E}_{0}  \left[\int_{\mathcal{J}^c} e^{-\lambda t} f(\widetilde{X}_t) dt\right] \notag \\
		&=\frac{\int_0^\infty {R}_\lambda^D f(y) \Pi^\Psi(dl)}{\Psi(\sqrt{\lambda+r}-\sqrt{r})} + \frac{1}{\Psi(\sqrt{\lambda + r} - \sqrt{r})} \int_0^\infty e^{-y(\sqrt{\lambda + r} + \sqrt{r})} f(y) dy \notag\\
		&= \frac{\int_0^\infty e^{-y(\sqrt{\lambda + r} + \sqrt{r})} f(y) dy + \int_0^\infty {R}_\lambda^D f(y) \Pi^\Psi(dl)}{\Psi(\sqrt{\lambda + r} - \sqrt{r})}.
	\end{align}
	We want to verify that \eqref{resolvent}, with $R_\lambda f(0)$ given in \eqref{resolvent:Xtilde0}, satisfies the boundary conditions
	\begin{align}
		\label{bc:resolvent}
		\frac{d}{dx} R_\lambda f(x) + \mathbf{D}_x^\Psi R_\lambda f(x)=0, \quad \text{ in } x=0.
	\end{align}
	From the construction of $R_\lambda^D f(x)$, we have
	\begin{align*}
		\frac{d}{dx} R_\lambda^D f(x) \big\vert_{x=0} &= \int_0^\infty e^{-y(\sqrt{\lambda + r} + \sqrt{r})} dy;\\
		\mathbf{D}_x^\Psi R_\lambda^D f(x)&= \lim_{x \to 0}\int_{0}^\infty (R^D_\lambda f(x+y) - R^D_\lambda f(x)) \Pi^\Psi(dy)\\
		&= \int_0^\infty R_\lambda^D f(y) \Pi^\Psi(dy).
	\end{align*}
	For the part from the hitting time, we observe that
	\begin{align*}
		&\frac{d}{dx} e^{-x(\sqrt{\lambda+ r} - \sqrt{r})} R_\lambda f(0) \big\vert_{x=0}= -(\sqrt{\lambda+ r} - \sqrt{r}) R_\lambda f(0);\\
		&\mathbf{D}_x^\Psi e^{-x(\sqrt{\lambda+ r} - \sqrt{r})} R_\lambda f(0) = R_\lambda f(0) \, \lim_{x \to 0} \int_0^\infty(e^{-(x+y)(\sqrt{\lambda+ r} - \sqrt{r})} - e^{-x(\sqrt{\lambda+ r} - \sqrt{r})} ) \Pi^\Psi(dy)\\
		&=- R_\lambda f(0) \int_{0}^\infty (1- e^{-x(\sqrt{\lambda+ r} - \sqrt{r})}) \Pi^\Psi(dy).
	\end{align*}
	From the linearity of the  operators at the boundary, we conclude that
	\begin{align*}
		&0=\left(\frac{d}{dx} R_\lambda f(x) + \mathbf{D}_x^\Psi R_\lambda f(x)\right)_{x=0}
		\\&=\int_0^\infty e^{-y(\sqrt{\lambda + r} + \sqrt{r})} dy + \int_0^\infty R_\lambda^D f(y) \Pi^\Psi(dy) \,+\\
		&\, - R_\lambda f(0) \left((\sqrt{\lambda+ r} - \sqrt{r}) + \int_{0}^\infty (1- e^{-x(\sqrt{\lambda+ r} - \sqrt{r})}) \Pi^\Psi(dy) \right)\\
		&=\int_0^\infty e^{-y(\sqrt{\lambda + r} + \sqrt{r})} dy + \int_0^\infty R_\lambda^D f(y) \Pi^\Psi(dy) - \Psi(\sqrt{\lambda+ r} - \sqrt{r})R_\lambda f(0).
	\end{align*}
	Then, the resolvent at zero, such that the boundary conditions are satisfied, is given by
	\begin{align*}
		R_\lambda f(0)=\frac{\int_0^\infty e^{-y(\sqrt{\lambda + r} + \sqrt{r})} f(y) dy + \int_0^\infty {R}_\lambda^D f(y) \Pi^\Psi(dl)}{\Psi(\sqrt{\lambda + r} - \sqrt{r})},
	\end{align*}
	which coincides with the one obtained in \eqref{resolvent:Xtilde0} for the process $\widetilde{X}$. Hence, from the uniqueness of the inverse of the Laplace transform, $u(t,x)$ is the solution of the \eqref{eq:NLBVP}.
\end{proof}
\begin{proof}[Proof of Theorem \ref{tm:local-time-pallino}]
	We use the notation of \hyperref[proof:NLBVP]{\textit{Proof of Theorem \ref{NLBVP}}}. Let \(\widetilde{\ell}\) be the local time of \(\widetilde{X}\), so defined as
	\begin{align*}
		\widetilde{\ell}_t := \lim_{\varepsilon \to 0}\frac{1}{2 \varepsilon} \int_0^t \mathbf{1}_{\{\widetilde{X}_s < \varepsilon\}} ds,
	\end{align*}
	and our aim is to prove that it is equivalent to \(L^\Psi \circ \widetilde{\gamma}\). Let us recall 
	\begin{align*}
		&\mathcal{J}^c:=\{t \geq 0 : H^\Psi L^\Psi_t = t\}\\
		&\mathcal{J} := \bigcup_{n \geq 1} \mathcal{J}_n= \bigcup_{n \geq 1} [l_n^{-}, l_n^+),
	\end{align*}
	and we define 
	\begin{align*}
		\widetilde{D}:= \{t \geq 0 : \widetilde{\gamma}_t \in \mathcal{J}^c\}.
	\end{align*}
	From the definition of $\mathcal{J}$, we obtain that the complement of \(\widetilde{D}\), denoted by \(\widetilde{D}^c\), is
	\begin{align*}
		\widetilde{D}^c:= \bigcup_{n \geq 1} [\widetilde{\gamma}^{-1}_{-}(l_n^{-}), \widetilde{\gamma}^{-1}_{-}(l_n^{+})),
	\end{align*}
	where, as before, \(\widetilde{\gamma}^{-1}_{-}\) is the left-inverse of \(\widetilde{\gamma}\). Since the boundary is shared between \(\widetilde{D}\) and \(\widetilde{D}^c\), we have that \(\partial \widetilde{D}\) is countable (if it were not, its complement could not be countable either). From the continuity of \(\widetilde{\gamma}\) and the fact that \(\widetilde{X}=\widetilde{B}\) on \(\widetilde{D}\), where the process does not jump, we write
	\begin{align}
		\label{pezzo+}
		\lim_{\varepsilon \to 0}\frac{1}{2 \varepsilon} \int_0^t \mathbf{1}_{\{\widetilde{X}_s < \varepsilon, \, s \in \widetilde{D} \cap [0,t]\}} ds &= \lim_{\varepsilon \to 0}\frac{1}{2 \varepsilon} \int_0^t \mathbf{1}_{\{\widetilde{B}_s < \varepsilon, \, s \in \widetilde{D} \cap [0,t]\}} ds \notag\\
		&= \widetilde{\gamma}_{\widetilde{D} \cap [0,t]}.
	\end{align}
	Let us now understand what happens outside \(\widetilde{D}\). We define \(\widetilde{D}_\delta^- := \{t \geq 0 : t \notin \widetilde{D}, \, \widetilde{X} < \delta\}\), for $\delta>0$, and we note
	\begin{align*}
		\widetilde{D}_\delta^-&:= \{t \geq 0 : t \notin \widetilde{D}, \, \widetilde{X} < \delta\}\\
		&=\bigcup_{n \geq 1} [\widetilde{\gamma}^{-1}_{-}(l_n^{-}), \widetilde{\gamma}^{-1}_{-}(l_n^{+})) \cap \{t \geq 0: l_n^+ - \widetilde{\gamma} < \delta\}\\
		&=\bigcup_{n \geq 1} [\widetilde{\gamma}^{-1}_{-}(l_n^{-}), \widetilde{\gamma}^{-1}_{-}(l_n^{+})) \cap [\widetilde{\gamma}^{-1}_{-}(l_n^+ - \delta), + \infty)\\
		&=\bigcup_{n \geq 1} [\max\left\{\widetilde{\gamma}^{-1}_{-}(l_n^-), \widetilde{\gamma}^{-1}_{-}(l_n^+ - \delta)\right\}, \widetilde{\gamma}^{-1}_{-}(l_n^+)).
	\end{align*}
	Since \(\gamma^{-1}_-\) is left-continuous, we have that \(\cap_{\delta>0} \widetilde{D}_\delta^- = \emptyset\) which guarantees that \(\partial \widetilde{D}_\delta^-\) is countable. From the continuity of \(\widetilde{\gamma}\), we obtain
	\begin{align}
		\label{pezzo-}
		\limsup_{\varepsilon \to 0} \frac{1}{2 \varepsilon} \int_0^t \mathbf{1}_{\{\widetilde{X}_s < \varepsilon, \, s \in [0,t] \setminus \widetilde{D}\}} \, ds 
		&\leq \limsup_{\varepsilon \to 0} \frac{1}{2 \varepsilon} \int_0^t \mathbf{1}_{\{\widetilde{B}_s < \varepsilon, \, s \in \widetilde{D}^-_\delta  \cap [0,t]\}} \, ds \notag \\
		&= \widetilde{\gamma}_{\widetilde{D}^-_\delta \cap [0,t]} \xrightarrow{\delta \downarrow 0} 0.
	\end{align}
	By collecting \eqref{pezzo+} and \eqref{pezzo+}, we have that
	\begin{align*}
		\widetilde{\ell}_t=\widetilde{\gamma}_{\widetilde{D} \cap [0,t]}.
	\end{align*}
	From the definition of \(\widetilde{D}\), we see
	\begin{align*}
		\widetilde{\gamma}_{\widetilde{D} \cap [0,t]}= \big\vert \mathcal{J}^c \cap [0, \widetilde{\gamma}_t] \big\vert.
	\end{align*}
	But, on \(\mathcal{J}^c\), since the subordinator does not jump, we have \(dL^\Psi_t = dt\), which concludes the characterization of the local time 
	\begin{align*}
		\widetilde{\ell}_t=\widetilde{\gamma}_{\widetilde{D} \cap [0,t]}= \big\vert \mathcal{J}^c \cap [0, \widetilde{\gamma}_t] \big\vert= L^\Psi \circ \widetilde{\gamma}_t.
	\end{align*}
\end{proof}

\section*{Acknowledgments}

F.C. and M.D. thank Sapienza and the group INdAM-GNAMPA for the support.  
The research has been mostly funded by MUR under the project PRIN 2022 - 2022XZSAFN - CUP B53D23009540006 - PNRR M4.C2.1.1.: Anomalous Phenomena on Regular and Irregular Domains:
Approximating Complexity for the Applied Sciences. \\
Web Site: \url{https://www.sbai.uniroma1.it/~mirko.dovidio/prinSite/index.html} \\
F.C. wishes to thank BCAM-Basque Center for Applied Mathematics for the hospitality during the visiting period from March 12, 2025 to June 11, 2025. \\
G.P. is supported by the Basque Government through the BERC 2022--2025 program and by the Ministry of Science and Innovation: BCAM Severo Ochoa accreditation CEX2021-001142-S / MICIN / AEI / 10.13039/501100011033.


\begin{thebibliography}{40}
	
	\bibitem{evansreview}
	M.R. Evans, S.N. Majumdar, and G. Schehr.
	\newblock Stochastic resetting and applications.
	\newblock {\em Journal of Physics A}, 53(19):193001, 2020.
	
	\bibitem{sde-resetting}
	M. Magdziarz and K. Ta\'zbierski.
	\newblock Stochastic representation of processes with resetting.
	\newblock {\em Physical Review E}, 106(1):014147, 2022.
	
	\bibitem{boncoldovpag2024}
	S. Bonaccorsi, F. Colantoni, M. D’Ovidio, and G. Pagnini.
	\newblock Non-local boundary value problems, stochastic resetting and Brownian motions on graphs.
	\newblock {\em arXiv preprint} arXiv:2209.14135, 2024.
	
	\bibitem{colantoni2023non}
	F. Colantoni.
	\newblock Non-local skew and non-local skew sticky Brownian motions.
	\newblock {\em Journal of Evolution Equations}, 25:39, 2025.
	
	\bibitem{fluctuation}
	U. Seifert.
	\newblock Stochastic thermodynamics, fluctuation theorems and molecular machines.
	\newblock {\em Reports on Progress in Physics}, 75(12):126001, 2012.
	
	\bibitem{reverse-diffusion1}
	U.G. Haussmann and E. Pardoux.
	\newblock Time reversal of diffusions.
	\newblock {\em Annals of Probability}, 14(4):1188–1205, 1986.
	
	\bibitem{reverse-diffusion2}
	A. Millet, D. Nualart, and M. Sanz.
	\newblock Integration by parts and time reversal for diffusion processes.
	\newblock {\em Annals of Probability}, 17(1):208–238, 1989.
	
	\bibitem{reverse-reflected}
	F. Petit.
	\newblock Time reversal and reflected diffusions.
	\newblock {\em Stochastic Processes and their Applications}, 69(1):25–53, 1997.
	
	\bibitem{conforti22}
	G. Conforti and C. L\'eonard.
	\newblock Time reversal of Markov processes with jumps under a finite entropy condition.
	\newblock {\em Stochastic Processes and their Applications}, 144:85–124, 2022.
	
	\bibitem{reverse-jumps}
	M. Hutzenthaler and J.E. Taylor.
	\newblock Time reversal of some stationary jump diffusion processes from population genetics.
	\newblock {\em Advances in Applied Probability}, 42(4):1147–1171, 2010.
	
	\bibitem{bass-sde-jumps}
	R.F. Bass.
	\newblock Stochastic differential equations with jumps.
	\newblock {\em Probability Surveys}, 1:1--19, 2004.
	
	\bibitem{pal2017integral}
	A. Pal and S. Rahav.
	\newblock Integral fluctuation theorems for stochastic resetting systems.
	\newblock {\em Physical Review E}, 96(6):062135, 2017.
	
	\bibitem{generators}
	M. John and Y. Wu.
	\newblock Calculating infinitesimal generators.
	\newblock {\em Journal of Stochastic Analysis}, 2(4):4--16, 2021.
	
	\bibitem{bass-adding}
	R.F. Bass.
	\newblock Adding and subtracting jumps from Markov processes.
	\newblock {\em Transactions of the American Mathematical Society}, 255:363--376, 1979.
	
	\bibitem{sandev-iomin}
	T. Sandev and A. Iomin.
	\newblock {\em Special Functions of Fractional Calculus --- Applications to Diffusion and Random Search Processes}.
	\newblock World Scientific, Hackensack, NJ, 2023.
	
	\bibitem{evans2014}
	M.R. Evans and S.N. Majumdar.
	\newblock Diffusion with resetting in arbitrary spatial dimension.
	\newblock {\em Journal of Physics A}, 47(28):285001, 2014.
	
	\bibitem{ItoMckean}
	K. Itô and H.P. McKean Jr.
	\newblock Brownian motions on a half line.
	\newblock {\em Illinois Journal of Mathematics}, 7:181--231, 1963.
	
	\bibitem{feller}
	W. Feller.
	\newblock The parabolic differential equations and the associated semi-groups of transformations.
	\newblock {\em Annals of Mathematics}, 55:468--519, 1952.
	
	\bibitem{bertoin1999subordinators}
	J. Bertoin.
	\newblock Subordinators: examples and applications.
	\newblock In {\em Lectures on Probability Theory and Statistics, Saint-Flour 1997}, Lecture Notes in Mathematics, vol. 1717, pages 1--91. Springer, 1999.
	
	\bibitem{schilling2012bernstein}
	R.L. Schilling, R. Song, and Z. Vondraček.
	\newblock {\em Bernstein Functions. Theory and Applications}, 2nd edition.
	\newblock De Gruyter, Berlin, 2012.
	
	\bibitem{Meerschaert-triangular}
	M.M. Meerschaert and H.-P. Scheffler.
	\newblock Triangular array limits for continuous time random walks.
	\newblock {\em Stochastic Processes and their Applications}, 118(9):1606--1633, 2008.
	
	\bibitem{sato}
	K.-i. Sato.
	\newblock {\em Lévy Processes and Infinitely Divisible Distributions}, Revised edition.
	\newblock Cambridge University Press, 2013.
	
	\bibitem{kilbas}
	S.G. Samko, A.A. Kilbas, and O.I. Marichev.
	\newblock {\em Fractional Integrals and Derivatives}.
	\newblock Gordon and Breach Science Publishers, Yverdon, 1993.
	
	\bibitem{markov-book}
	S.N. Ethier and T.G. Kurtz.
	\newblock {\em Markov Processes: Characterization and Convergence}.
	\newblock Wiley, New York, 1986.
	
	\bibitem{nagasawa-quantum}
	M. Nagasawa.
	\newblock Time reversal of Markov processes and relativistic quantum theory.
	\newblock {\em Chaos, Solitons and Fractals}, 8(11):1711--1772, 1997.
	
	\bibitem{nelson}
	E. Nelson.
	\newblock The adjoint Markoff process.
	\newblock {\em Duke Mathematical Journal}, 25:671--690, 1958.
	
	\bibitem{ness-stable}
	T. Grzywny, K. Szczypkowski, Z. Palmowski, and B. Trojan.
	\newblock Stationary states for stable processes with partial resetting.
	\newblock {\em arXiv preprint} arXiv:2412.15626, 2024.
	
	\bibitem{nagasawa-reflecting}
	M. Nagasawa.
	\newblock The adjoint process of a diffusion with reflecting barrier.
	\newblock {\em Kodai Mathematical Seminar Reports}, 13:235--248, 1961.
	
	\bibitem{nasawa-markov}
	M. Nagasawa.
	\newblock Time reversions of Markov processes.
	\newblock {\em Nagoya Mathematical Journal}, 24:177--204, 1964.
	
	\bibitem{revuz-yor}
	D. Revuz and M. Yor.
	\newblock {\em Continuous Martingales and Brownian Motion}, 3rd edition.
	\newblock Springer, 1999.
	
	\bibitem{BorodinSalminen}
	A.N. Borodin and P. Salminen.
	\newblock {\em Handbook of Brownian Motion — Facts and Formulae}, 2nd edition.
	\newblock Birkhäuser, Basel, 2002.
	
	\bibitem{singh2022}
	P. Singh and A. Pal.
	\newblock First-passage Brownian functionals with stochastic resetting.
	\newblock {\em Journal of Physics A}, 55(23):234001, 2022.
	
	\bibitem{localtime}
	A. Pal, R. Chatterjee, S. Reuveni, and A. Kundu.
	\newblock Local time of diffusion with stochastic resetting.
	\newblock {\em Journal of Physics A}, 52(26):264002, 2019.
	
	\bibitem{BlumenthalGetoor}
	R.M. Blumenthal and R.K. Getoor.
	\newblock Local times for Markov processes.
	\newblock {\em Zeitschrift für Wahrscheinlichkeitstheorie und Verwandte Gebiete}, 3:50--74, 1964.
	
	\bibitem{bertoinlevy}
	J. Bertoin.
	\newblock {\em Lévy Processes}.
	\newblock Cambridge University Press, 1996.
	
	\bibitem{kwasnicki}
	M. Kwaśnicki.
	\newblock Boundary traces of shift-invariant diffusions in half-plane.
	\newblock {\em Annales de l’Institut Henri Poincaré, Probabilités et Statistiques}, 59(1):411--436, 2023.
	
	\bibitem{Pilipenko}
	A. Pilipenko and A. Sarantsev.
	\newblock Boundary approximation for sticky jump-reflected processes on the half-line.
	\newblock {\em Electronic Journal of Probability}, 29:32, 2024.
	
	\bibitem{ikeda2014stochastic}
	N. Ikeda and S. Watanabe.
	\newblock {\em Stochastic Differential Equations and Diffusion Processes}.
	\newblock North-Holland, Amsterdam–New York; Kodansha, Tokyo, 1981.
	
	\bibitem{mori2023entropy}
	F. Mori, K.S. Olsen, and S. Krishnamurthy.
	\newblock Entropy production of resetting processes.
	\newblock {\em Physical Review Research}, 5(2):023103, 2023.
	
	\bibitem{olsen2024thermodynamic}
	K.S. Olsen, D. Gupta, F. Mori, and S. Krishnamurthy.
	\newblock Thermodynamic cost of finite-time stochastic resetting.
	\newblock {\em Physical Review Research}, 6(3):033343, 2024.
	
\end{thebibliography}
\end{document}